\documentclass[11pt]{article}
\usepackage{amsmath,amsthm,amssymb,color}
\usepackage{hyperref}
\usepackage{fullpage}

\usepackage{tikz}

\newtheorem{thm}{Theorem}[section]
\newtheorem{claim}[thm]{Claim}
\newtheorem{lem}[thm]{Lemma}
\newtheorem{define}[thm]{Definition}

\newtheorem{obs}[thm]{Observation}

\newtheorem{prop}[thm]{Proposition}

\newtheorem*{boundedmaximumdegreerestatement}{Theorem~\ref{thm:main-bounded-degree}}
\newtheorem*{boundedaveragedegreerestatement}{Theorem~\ref{thm:main-sparse}}
\newtheorem*{fixedexcludedminorrestatement}{Theorem~\ref{thm:main-fixed-minor}}
\newtheorem*{randomoutrestatement}{Theorem~\ref{thm:main-random-tradeoff}}

\let\textunderscorechar\_

\def\R{{\mathbb{R}}}

\def\P{{\mathbb{P}}}

\def\cT{{\cal T}}

\def\u{{\mathcal{u}}}
\def\v{{\mathcal{v}}}

\def\_{\,\,\,\,\,}

\def\supp{\textsf{supp}}

\newcommand{\eps}{\epsilon}

\newcommand{\remove}[1]{}

\def\RP{{\mathbb{RP}}}
\DeclareMathOperator{\aff}{aff}
\DeclareMathOperator{\affdim}{affdim}

\DeclareMathOperator{\SGdim}{SG\text{-}dim}
\DeclareMathOperator{\Ord}{Ord}

\begin{document}

\title{The Sylvester--Gallai dimension of graphs}
\author{Zeev Dvir\thanks{Department of Mathematics  and Department of Computer Science,
Princeton University.
Email: \texttt{zdvir@princeton.edu}.}}

\date{}
\maketitle

\begin{abstract}
For an undirected graph $G$, its \emph{Sylvester--Gallai dimension} $\SGdim(G)$ is the largest affine dimension of a configuration of distinct real points indexed by $V(G)$ in which every line determined by an edge of $G$ is a special line (contains at least three points). Hence, the classical Sylvester--Gallai theorem can be stated as $\SGdim(K_n)=1$ for the complete graph $K_n$. We initiate the systematic study of this new graph parameter and prove lower bounds for certain graph families (bounded degree, sparse, minor-free) as well as an upper bound for random graphs.
\end{abstract}

\section{Introduction}\label{sec:introduction}

Let $P=\{p_1,\ldots,p_n\}$ be a set of $n\geq 3$ distinct points in $\R^d$. We write $\aff(P)$ for the affine span of $P$ and $\affdim(P)$ for the dimension of the affine span. A line determined by two points of $P$ is \emph{special} if it contains at least three points of $P$; otherwise it is called an \emph{ordinary} line. The classical Sylvester--Gallai theorem says that, unless all points of $P$ sit on a single line, there must be at least one ordinary line spanned by $P$. This can be restated as follows: if, for all pairs $1\leq i<j\leq n$, the line through $p_i,p_j$ is special, then the affine span of $P$ must be one-dimensional. In this paper we suggest studying the more general question in which we only know that a \emph{subset} of the pairs $i<j$ (modeled as a graph) determine special lines and we want to understand whether this forces $P$ to have small affine dimension. This motivates the following definition.

\begin{define}[Sylvester--Gallai dimension]
\label{def:sg-dimension}
Let $G=(V,E)$ be a simple graph on $n\geq 3$ vertices. A \emph{special-line realization} of $G$ is a configuration $
P=\{p_v:v\in V\}\subset\R^d
$
of distinct points such that, for every edge $uv\in E$, the line through $p_u$ and $p_v$ contains a third point of $P$. (The vertex corresponding to this third point need not be adjacent to either $u$ or $v$.) The \emph{Sylvester--Gallai dimension} of $G$, denoted $\SGdim(G)$, is
\[
\SGdim(G)=\max\{\affdim(P): P \subset \R^d \text{ is a special-line realization of }G\},
\]
where the ambient dimension $d$ can be arbitrarily large.
\end{define}

The Sylvester--Gallai theorem can now be stated as
\[
\SGdim(K_n)\leq 1,
\]
where $K_n$ is the complete graph on $n\geq 3$ vertices (it is in fact an equality, as the points are distinct). A generalization of this theorem for graphs with large minimum degree was first given in~\cite{BDWY}, with the following improvement appearing in~\cite{DSW}. The dimension bound given by this theorem is tight, up to the exact constant (take for example, a union of $n/D$ copies of $K_D$ and place them on lines in independent directions).

\begin{thm}[Quantitative Sylvester--Gallai]
\label{thm:quantitative-sg}
Let $G$ be a graph on $n\geq 3$ vertices with minimum degree at least $D\geq 1$. Then
\[
\SGdim(G)\leq\frac{12(n-1)}{D}.
\]
\end{thm}

In this work we complement these results by proving several new \emph{lower bounds} on the SG dimension of certain graph families, as well as an upper bound for random graphs which beats the one of Theorem~\ref{thm:quantitative-sg} by an additional factor of $D$ (for sufficiently large $D$).

\subsection{Our results}\label{subsec:our-results}

Our first result is a lower bound for graphs of bounded maximum degree. We conjecture that the $D^3$ factor could be replaced by $D^2$ (which would be tight, for sufficiently large $D$, by Theorem~\ref{thm:main-random-tradeoff}).
\begin{thm}[Bounded maximum degree]
\label{thm:main-bounded-degree}
Let $G$ be a graph on $n\geq 3$ vertices with maximum degree at most $D \geq 1$. Then
\[
\SGdim(G)=\Omega\left(\frac{n}{D^3}\right).
\]
\end{thm}

Our next theorem only assumes a bound on the \emph{average} degree of $G$. In this case, we are able to prove a much weaker bound. We do not know if this bound is tight or whether an $\Omega_D(n)$ bound could be proven here as well.

\begin{thm}[Bounded average degree]
\label{thm:main-sparse}
For every integer $D\geq 1$, every graph $G$ on $n\geq 3$ vertices with average degree at most $D$ satisfies
\[
\SGdim(G)=\Omega_D(n^{1/D}).
\]
\end{thm}

Our next theorem gives a linear lower bound for graphs excluding a fixed minor. These include forests, planar graphs, and graphs of bounded treewidth, among others.

\begin{thm}[Fixed excluded minor]
\label{thm:main-fixed-minor}
For every finite graph $H$, there is a constant $c_H>0$ such that every $H$-minor-free graph $G$ on $n\geq 3$ vertices satisfies
\[
\SGdim(G)\geq c_H n.
\]
\end{thm}

The value of the constant $c_H$ coming from our proof depends on constants from several other papers, and we make no attempt to make it explicit. For particular excluded minors one can derive explicit bounds. For example, for forests, a simple induction gives
\[
\SGdim(G)\geq\left\lfloor\frac{n-1}{2}\right\rfloor
\]
(we leave this as an exercise).

Our final result is an upper bound on the SG dimension of random graphs. It will be convenient to work with the following random graph model\footnote{The same theorem can also be derived for the more standard random $G(n,p)$ model (when $pn\gg\log n$) using a standard coupling argument (we omit the details).} 

\begin{define}[Random $\widehat G_{n,D}$ graph]
\label{def:symmetrized-random-out}
Let $n\geq 2$ and $1\leq D\leq n-1$ be integers. Independently for each $v\in[n]$, choose a uniformly random $D$-element set $N^+(v)\subseteq[n]\setminus\{v\}$. The \emph{symmetrized random $D$-out graph} $\widehat G_{n,D}$ is the graph on $[n]$ in which $u$ and $v$ are adjacent if $u\in N^+(v)$ or $v\in N^+(u)$.
\end{define}
Throughout the paper, all logarithms are to base $2$.
\begin{thm}[$\widehat G_{n,D}$ upper bound]
\label{thm:main-random-tradeoff}
There are absolute constants $A>0$ and $K>1$ such that, for all sufficiently large $n$ and all integers $K\log n\leq D\leq n-1$,
\[
\SGdim(\widehat G_{n,D})
\leq
\max\left\{2,\left\lceil\dfrac{An\log n}{D^2\log(D/\log n)}\right\rceil\right\}
\]
with probability $1-o(1)$ as $n\to\infty$.
\end{thm}

Notice that for the range $D<K\log n$, Theorem~\ref{thm:quantitative-sg} already  gives $\SGdim(\widehat G_{n,D})\leq 12(n-1)/D$. For larger $D$, however, Theorem~\ref{thm:main-random-tradeoff} is stronger (e.g., for $D=n^\eps$ it is roughly $n/D^2$ instead of $n/D$).
In particular, taking $D \sim \sqrt n$ it shows that there exists a graph $G$ with $O(n^{3/2})$ edges and $\SGdim(G)\leq 2$. We do not know, however, how to construct such a graph explicitly (say by a deterministic algorithm running in time polynomial in $n$), even if we replace $2$ with any other constant.

\subsection{Proof overview}\label{subsec:proof-overview}

The three lower bounds (for bounded maximum degree, bounded average degree, and minor-free graphs) all rely on the same basic construction, although with increasing complications. We start with the simplest case, when all degrees in $G$ are at most $D$. We first greedily find a set $X$ of $\Omega(n/D^2)$ vertices with disjoint neighborhoods. We construct a projective `book' composed of two-dimensional `pages' $\Pi_1,\Pi_2,\ldots$ that all intersect in a single line $L$ (the `spine'), so that each page contributes one additional dimension (see Section~\ref{sec:basic-book} for details). We start by placing all vertices not adjacent to any vertex of $X$ on the spine in arbitrary locations. We leave at least three distinct vertices on the spine, so the spine is a special line.

We now partition $X$ into $p=\Omega(n/D^3)$ `blocks' $X_1,\ldots,X_p$, each of size at least $2D+1$. We show how each block $X_i$ can be placed on page $\Pi_i$ (outside the spine) in a way that makes all edges touching $X_i$ special. The neighbors of $X_i$ (which we know are not neighbors of any other $X_j$) will be placed on the spine. We may identify the spine $L$ with the line at infinity of $\Pi_i$. Hence, we would like to place each $x\in X_i$ at a point $p_x\in\Pi_i\setminus L$ and place the neighbors $N_G(x)$ on $L$, so that each edge $xv$, with $v\in N_G(x)$, lies on a special line. Since $L$ is the line at infinity, this edge will be special precisely when there is another vertex $y\in X_i$ such that the line connecting $x$ and $y$ has direction $v$. Suppose we place the $2D+1$ points of $X_i$ generically in $\Pi_i\setminus L$, so that each pair of points has a distinct direction. We can now place the neighbors of $x$ on $L$ at positions corresponding to $D$ `private' directions (between $x$ and some other $D$ points of $X_i$). These private directions can be obtained by orienting the edges of the complete graph $K_{2D+1}$ so that each vertex has outdegree $D$. Figure~\ref{fig:bounded-degree-page} illustrates this construction on a single page.

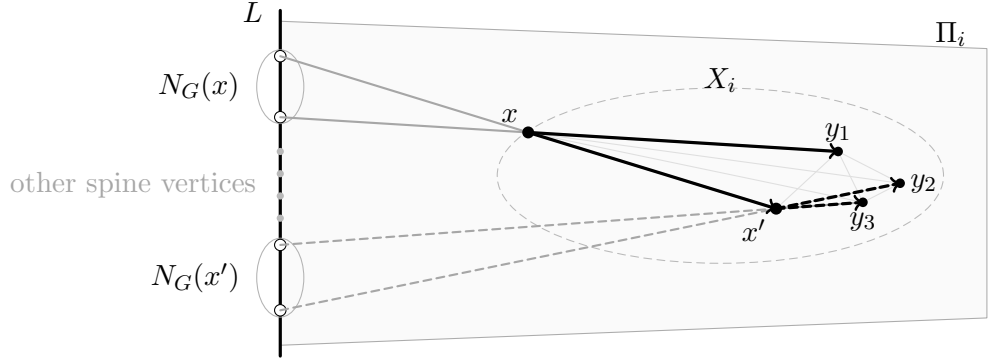
\begin{figure}[h!]
\centering
\begin{tikzpicture}[x=.82cm,y=.84cm,line cap=round,line join=round]
  \fill[gray!4] (1.8,-2.55)--(1.8,2.55)--(13.2,2.12)--(13.2,-2.12)--cycle;
  \draw[gray!65] (1.8,-2.55)--(1.8,2.55)--(13.2,2.12)--(13.2,-2.12)--cycle;
  \draw[very thick] (1.8,-2.72)--(1.8,2.72);
  \node[left] at (1.68,2.70) {$L$};
  \node at (12.62,2.35) {$\Pi_i$};

  \coordinate (vone) at (1.8,2.0);
  \coordinate (vtwo) at (1.8,1.04);
  \coordinate (wone) at (1.8,-2.0);
  \coordinate (wtwo) at (1.8,-.971);
  \draw[gray!60] (1.8,1.52) ellipse (.38 and .58);
  \draw[gray!60] (1.8,-1.486) ellipse (.38 and .62);
  \node[left] at (1.31,1.52) {$N_G(x)$};
  \node[left] at (1.31,-1.486) {$N_G(x')$};
  \foreach \p in {vone,vtwo,wone,wtwo} \draw[fill=white] (\p) circle (2.15pt);
  \foreach \y in {-.55,-.2,.15,.5} \fill[gray!55] (1.8,\y) circle (1.35pt);
  \node[left,gray!70] at (1.58,-.03) {other spine vertices};

  \coordinate (px) at (5.8,.8);
  \coordinate (pxp) at (9.8,-.4);
  \coordinate (yone) at (10.8,.5);
  \coordinate (ytwo) at (11.8,0);
  \coordinate (ythree) at (11.2,-.3);
  \draw[densely dashed,gray!55] (8.9,.12) ellipse (3.60 and 1.38);
  \node at (8.9,1.62) {$X_i$};
  \foreach \a/\b in {px/pxp,px/yone,px/ytwo,px/ythree,pxp/yone,pxp/ytwo,pxp/ythree,yone/ytwo,yone/ythree,ytwo/ythree} {
    \draw[gray!25] (\a)--(\b);
  }

  \draw[thick,gray!70] (vone)--(pxp);
  \draw[thick,gray!70] (vtwo)--(yone);
  \draw[thick,densely dashed,gray!70] (wone)--(ytwo);
  \draw[thick,densely dashed,gray!70] (wtwo)--(ythree);
  \draw[->,very thick] (px)--(pxp);
  \draw[->,very thick] (px)--(yone);
  \draw[->,very thick,densely dashed] (pxp)--(ytwo);
  \draw[->,very thick,densely dashed] (pxp)--(ythree);

  \fill (px) circle (2.25pt) node[above left] {$x$};
  \fill (pxp) circle (2.25pt) node[below left] {$x'$};
  \fill (yone) circle (1.8pt) node[above] {$y_1$};
  \fill (ytwo) circle (1.8pt) node[right] {$y_2$};
  \fill (ythree) circle (1.8pt) node[below] {$y_3$};
\end{tikzpicture}
\caption{One page of the bounded-degree construction, shown schematically for $D=2$. The block $X_i$ is placed generically in $\Pi_i\setminus L$. After orienting the complete graph on $X_i$ with outdegree $D$, each arrow $x\to y$ supplies a private direction for $x$. A vertex $v\in N_G(x)$ is placed at the corresponding point of $L$, so $v,x,y$ are collinear and the edge $xv$ lies on a special line. The solid and dashed lines show the directions assigned to $x$ and $x'$, respectively; the remaining tournament arrows are suppressed. Since the neighborhoods of vertices in $X$ are disjoint, all assigned points on $L$ are distinct.}
\label{fig:bounded-degree-page}
\end{figure}

We now switch to the case of graphs with \emph{average degree} at most $D$. Here, we cannot find an independent set $X$ with disjoint neighborhoods as before. Instead, using the sunflower lemma, we find an independent set $X$ of size roughly $n^{1/D}$ such that the neighborhoods $N_G(x)$, $x\in X$, form a sunflower with core $K$ and petals $P_x$ (i.e., each $N_G(x)=K\sqcup P_x$ and the sets $P_x$ are disjoint). One needs to be careful with vertices whose neighborhoods are identical (see the full proof for details). The former book construction handles the special case of disjoint neighborhoods (empty $K$) and can be extended to the case of nonempty $K$ as follows. Suppose $|K|=a$ and partition $X$ as before, but now require each $X_i$ to have size $2a(2D+1)$. We place all vertices nonadjacent to $X$ on the spine $L$ and also place all vertices of $K$ at special points $w_1,\ldots,w_a\in L$ (to be chosen later). Identifying $L$ as the line at infinity, we now wish to place each $x\in X_i$ in $\Pi_i\setminus L$ so that (1) $x$ lies on lines in all directions $w_1,\ldots,w_a$, and (2) $x$ lies on $|P_x|$ lines in `private' directions. This can be done by placing $2D+1$ copies (scaled and shifted generically) of a regular $2a$-gon. See Lemma~\ref{lem:free-core-book} (and  Figure~\ref{fig:free-core-book}) for details.

The minor-free theorem reduces, via standard tools from the literature on minor-free graphs (in particular that of $r$-divisions), to a  version of the preceding `sunflower book' construction in which the neighborhoods in each $X_i$ form a sunflower with core $K_i$ and all the petals are disjoint, even between different blocks. The cores might intersect, which creates a difficulty in the previous argument, since we cannot place the vertices of $K_i$ at the special directions coming from the regular polygon. To solve this problem, we replace the regular $2a$-gon with a projected `hypercube' of size $2^a$. This cube has the property that each vertex lies on $a$ lines in some fixed $a$ directions. See Lemma~\ref{lem:fixed-core-page} for details.

The upper bound for random graphs (Theorem~\ref{thm:main-random-tradeoff}) has two ingredients. The first is an upper bound on the number of `ordinary graphs', which are graphs arising from ordinary pairs  (pairs defining ordinary lines) in some point configuration $P \subset \R^d$. One can use known bounds on the zero patterns of low degree polynomial maps to show that the number of such graphs on $n$ vertices is asymptotically at most $\exp(n \log(n))$. The second ingredient is a lower bound on the number of ordinary pairs in sufficiently high dimensional  point sets $P$ (and hence on the number of edges in the corresponding ordinary graph $G$). This lower bound uses the design-matrix machinery from \cite{BDWY,DSW} as well as a result from \cite{deZeeuw}. This implies a strong probability upper bound for the event that a random graph `misses' all edges in some fixed ordinary graph $G$ coming from a high dimensional point set $P$. Combining these two ingredients, we can perform a union bound and show that a random graph will `hit' at least one edge in each high dimensional ordinary graph w.h.p. Such a graph cannot have a special line realization in the same dimension since this will force its edges to be disjoint from the edges in the ordinary graph for the same configuration.

\subsection{Paper organization}

Section~\ref{sec:basic-book} develops the geometric  constructions used throughout the paper. Section~\ref{sec:sparse} proves Theorems~\ref{thm:main-bounded-degree} (bounded degree graphs)   and~\ref{thm:main-sparse} (graphs with bounded average degree). Section~\ref{sec:minor-free} proves Theorem~\ref{thm:main-fixed-minor} (minor free graphs) and Section~\ref{sec:random-out} proves Theorem~\ref{thm:main-random-tradeoff} (random graphs).

\subsection{Acknowledgements}
Commercially available AI tools (ChatGPT 5.5 and 5.6) were used in all stages of research, including proof development and manuscript preparation. The author is fully responsible for the contents of the paper. Research supported by NSF grant  DMS-2246682.

\section{Geometric constructions}\label{sec:basic-book}
In this section we develop the necessary geometric machinery used in the proofs of our three lower bound theorems. We start with some basic gadgets in the first sub-section and then continue with the two main lemmas in the following two sub-sections.

\subsection{Secant gadgets}\label{sec:secantgad}
For a finite set of points $X\subset\P\R^d$, the \emph{secants} of $X$ are the lines connecting at least two distinct points of $X$. We will be interested in constructing sets $X$ under prescribed conditions on secant directions. For example, given a fixed set of directions $S$, we may want every point of $X$ to belong to a secant in each direction in $S$, while avoiding certain directions or requiring some secant directions to be unique to the configuration.\footnote{Closely related realization questions, in which directions are assigned to the edges of a graph, are studied under the names \emph{direction networks} and \emph{parallel drawings} see e.g.~\cite{Whiteley1988,Whiteley1989,Whiteley1996,StreinuTheran2010}.  Our setting differs in that some directions are prescribed while others must be globally distinct and avoid a finite forbidden set.}

The first construction  gives a simple gadget (the regular $2k$-gon) in which every one of the $2k$ points belongs to secants of $k$ fixed directions.

\begin{lem}[Polygon directions]\label{lem:polygon-directions}
For every integer $k\geq 1$, there are sets of directions $\mathcal D\subset\RP^1$ and points $Z\subset\R^2$, with $|\mathcal D|=k$ and $|Z|=2k$, such that every point of $Z$ belongs to a secant of $Z$ in each direction in $\mathcal D$.
\end{lem}

\begin{proof}
Let
\[
z_i=\left(\cos\frac{\pi i}{k},\sin\frac{\pi i}{k}\right),
\qquad i\in\mathbb Z/(2k)\mathbb Z,
\]
be the vertices of a regular $2k$-gon. For every odd residue $s$, let
\[
u_s=\left(-\sin\frac{\pi s}{2k},\cos\frac{\pi s}{2k}\right).
\]
The involution $i\mapsto s-i$ has no fixed point, and
\[
z_{s-i}-z_i
=2\sin\frac{\pi(s-2i)}{2k}\,u_s.
\]
The scalar is nonzero because $s-2i$ is odd. Thus every $z_i$ has a partner in direction $[u_s]$. The $k$ odd residues give distinct unoriented directions, so
\[
\mathcal D=\{[u_s]:s\text{ odd}\},
\qquad
Z=\{z_i:i\in\mathbb Z/(2k)\mathbb Z\}.
\]
\end{proof}

The second gadget lemma takes a given set $Q$ (such as the regular polygon) and shows that one can generate $m$ homothetic copies of $Q$ (obtained from $Q$ by scaling and shifting) in such a way that, for each point, the directions from that point to points in \emph{other} copies of $Q$ are unique. Thus, in the union of all copies, each point lies on secants satisfying the direction requirement of $Q$ and also lies on at least $(m-1)|Q|$ `private' cross-copy secant directions (ones not shared by other secants connecting points in two copies). In addition, we require that the resulting set of private directions avoid an arbitrary `forbidden' set $\mathcal F$.

For a nonzero vector $w\in\R^2$, we write $[w]\in\RP^1$ for its projective class, or equivalently its unoriented direction.

\begin{lem}[Generic homothetic copies]\label{lem:generic-homothetic}
Let $Q\subset\R^2$ and $\mathcal F\subset\RP^1$ be finite. For every $m\geq 1$, there are pairwise disjoint homothetic copies $Q_1,\ldots,Q_m$ of $Q$ such that all secants with endpoints in distinct copies have pairwise distinct directions outside $\mathcal F$. More formally, if
\[
\mathcal W
=
\bigl\{[x-y]:x\in Q_h,\ y\in Q_k,\ 1\leq h<k\leq m\bigr\},
\]
then
\[
|\mathcal W|=\binom{m}{2}|Q|^2
\qquad\text{and}\qquad
\mathcal W\cap\mathcal F=\emptyset.
\]
\end{lem}

\begin{proof}
Write $Q=\{v_1,\ldots,v_s\}$ and seek copies of the form
\[
x_{h,i}=t_h+\lambda_h v_i,
\qquad
Q_h=\{x_{h,1},\ldots,x_{h,s}\},
\]
where $t_h\in\R^2$ and $\lambda_h\neq 0$. For a cross-copy secant indexed by $e=(h,i;k,j)$, with $h<k$, put
\[
w_e=x_{k,j}-x_{h,i}=(t_k-t_h)+\lambda_k v_j-\lambda_h v_i.
\]

We will now show that one can express each collision, forbidden direction, or equality between two cross-copy directions as the zero set of a nonzero polynomial in the parameters $t_h,\lambda_h$. A collision between different copies is governed by the nonzero polynomial $\lVert x_{h,i}-x_{k,j}\rVert^2$. For $[u]\in\mathcal F$, the condition that $w_e$ has direction $[u]$ is $\det(w_e,u)=0$, again a nonzero polynomial because $t_k-t_h$ varies freely.

It remains to compare two distinct secants $w_e,w_f$ so that $e$ connects the two copies $h<k$ and $f$ connects the two copies $h'<k'$. If the two pairs are different ($h\neq h'$ or $k\neq k'$), then the two translation differences can be varied independently, and so $\det(w_e,w_f)$ is not identically zero.

If both secants join the same copies ($h=h'$ and $k=k'$), write
\[
w_e=T+A,
\qquad
w_f=T+A',
\qquad T=t_k-t_h.
\]
Then
\[
A'-A
=\lambda_k(v_{j'}-v_j)-\lambda_h(v_{i'}-v_i).
\]
Since the secants are distinct, at least one of the differences $(v_{j'}-v_j)$ or $(v_{i'}-v_i)$ is nonzero. Hence, we can choose $\lambda_k,\lambda_h$ for which $A'-A$ is nonzero and, for this choice, there will always be shifts $t_k,t_h$ that make $\det(T+A,T+A')$ nonzero.

Since there are only finitely many nontrivial polynomial conditions (including $\lambda_h \neq 0$ for all $h$), we may choose parameters outside their union. The copies are then disjoint, all cross-copy directions avoid $\mathcal F$, and these directions are pairwise distinct.
\end{proof}

\subsection{Basic book construction using sunflowers}

The constructions to follow all use the same geometric book construction. Its spine is a projective line $L$, and its pages are projective planes $\Pi_1,\ldots,\Pi_q$ that meet pairwise exactly in $L$ and together span $\RP^{q+1}$. This can be described in homogeneous coordinates $[x_0:x_1:z_1:\cdots:z_q]$ on $\RP^{q+1}$ by setting
\[
L=\{z_1=\cdots=z_q=0\},
\qquad
\Pi_i=\{z_j=0\text{ for every }j\neq i\}.
\]
Then each $\Pi_i$ is a projective plane, $\Pi_i\cap\Pi_j=L$ for $i\neq j$, and the pages span $\RP^{q+1}$. For $q=0$, the book consists only of its spine.

Recall that a family $(A_x)_{x\in X}$ is a \emph{sunflower} with core $K$ if $A_x=K\sqcup P_x$ for pairwise disjoint petals $P_x$. The next lemma is our first main geometric construction (used in the proofs of Theorems~\ref{thm:main-bounded-degree} and~\ref{thm:main-sparse}). It takes as input a graph $J$ with an independent set $X$ such that the neighborhoods of the vertices $x\in X$ have size at most $D$ and form a sunflower. The lemma then proves a lower bound on $\SGdim(J)$ by placing the vertices outside $X$ on the spine $L$ of the book and the vertices of $X$ on $\Omega(|X|/D)$ independent pages. The block of vertices belonging to a single page will be constructed using the gadgets of Lemmas~\ref{lem:generic-homothetic} and~\ref{lem:polygon-directions}, where the spine $L$ represents the line at infinity. In this way, the secants in a given direction correspond to special lines, as long as a vertex of $J\setminus X$ is placed on the spine at the point corresponding to that direction. We refer to this lemma as the `free-core sunflower book' lemma because we are allowed to place the sunflower core (a set of vertices in $J\setminus X$) anywhere on the spine $L$. We choose these positions to correspond to the directions occurring in Lemma~\ref{lem:polygon-directions} (the regular polygon of the appropriate size). Later we will encounter situations in which the core positions are predetermined (see Lemma~\ref{lem:fixed-core-page}), which leads to worse parameters.

Figure~\ref{fig:free-core-book} illustrates the construction underlying the proof of Lemma~\ref{lem:free-core-book}.

\begin{lem}[Free-core sunflower book]\label{lem:free-core-book}
Let $J$ be a graph on at least three vertices and let $X\subseteq V(J)$ be independent. Suppose that $(N_J(x))_{x\in X}$ is a sunflower and $d_J(x)\leq D$ for $x\in X$, where $D\geq 1$. Then
\[
\SGdim(J)>\frac{|X|}{10D}.
\]
\end{lem}

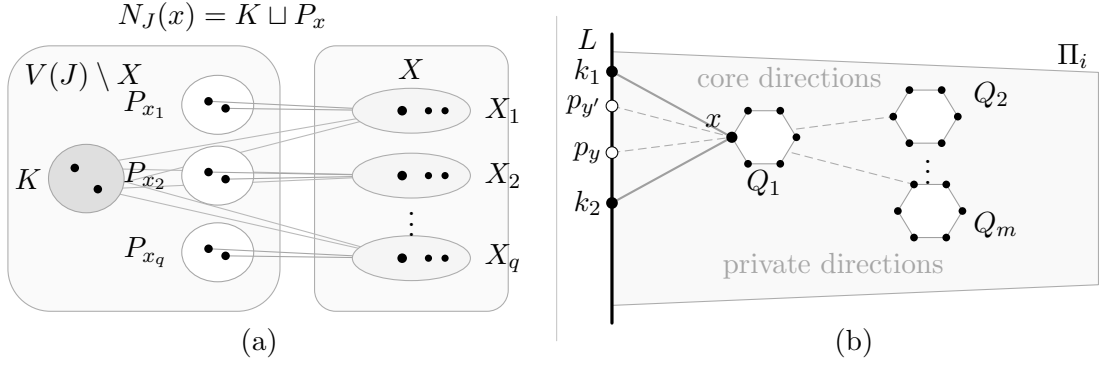
\begin{figure}[t]
\centering
\begin{tikzpicture}[x=.61cm,y=.78cm,line cap=round,line join=round]
  \draw[rounded corners=16pt,fill=gray!4,draw=gray!65] (-11.9,-2.25) rectangle (-6.0,2.25);
  \node at (-10.25,1.72) {$V(J)\setminus X$};
  \draw[rounded corners=8pt,fill=gray!4,draw=gray!65] (-5.25,-2.25) rectangle (-1.05,2.25);
  \node at (-3.15,1.86) {$X$};
  \node at (-7.25,2.75) {$N_J(x)=K\sqcup P_x$};

  \coordinate (kone) at (-10.45,.18);
  \coordinate (ktwo) at (-9.95,-.18);
  \coordinate (pone) at (-7.35,1.25);
  \coordinate (ptwo) at (-7.35,.05);
  \coordinate (pq) at (-7.35,-1.25);
  \coordinate (ponea) at (-7.55,1.31);
  \coordinate (poneb) at (-7.18,1.19);
  \coordinate (ptwoa) at (-7.55,.11);
  \coordinate (ptwob) at (-7.18,-.01);
  \coordinate (pqa) at (-7.55,-1.19);
  \coordinate (pqb) at (-7.18,-1.31);
  \coordinate (xone) at (-3.35,1.15);
  \coordinate (xtwo) at (-3.35,.05);
  \coordinate (xq) at (-3.35,-1.35);

  \foreach \u in {xone,xtwo,xq} {
    \draw[gray!55] (kone)--(\u);
    \draw[gray!55] (ktwo)--(\u);
  }
  \foreach \p in {pone,ptwo,pq} {
    \draw[fill=white,draw=gray!70] (\p) ellipse (.78 and .50);
  }
  \foreach \p in {ponea,poneb} \draw[gray!70] (\p)--(xone);
  \foreach \p in {ptwoa,ptwob} \draw[gray!70] (\p)--(xtwo);
  \foreach \p in {pqa,pqb} \draw[gray!70] (\p)--(xq);

  \draw[fill=gray!25,draw=gray!70] (-10.2,0) ellipse (.82 and .58);
  \fill (kone) circle (1.6pt);
  \fill (ktwo) circle (1.6pt);
  \node[left] at (-10.95,0) {$K$};
  \foreach \p in {ponea,poneb,ptwoa,ptwob,pqa,pqb} \fill (\p) circle (1.6pt);
  \node[left] at (-8.22,1.25) {$P_{x_1}$};
  \node[left] at (-8.22,.05) {$P_{x_2}$};
  \node[left] at (-8.22,-1.25) {$P_{x_q}$};

  \foreach \y/\lab/\xname in {1.15/$X_1$/xone,.05/$X_2$/xtwo,-1.35/$X_q$/xq} {
    \draw[fill=gray!8,draw=gray!65] (-3.15,\y) ellipse (1.28 and .38);
    \fill (\xname) circle (1.8pt);
    \fill (-2.78,\y) circle (1.35pt);
    \fill (-2.42,\y) circle (1.35pt);
    \node[right] at (-1.78,\y) {\lab};
  }
  \node at (-3.15,-.64) {$\vdots$};

  \draw[gray!45] (0,-2.75)--(0,2.75);
  \fill[gray!5] (1.2,-2.15)--(1.2,2.15)--(11.75,1.8)--(11.75,-1.8)--cycle;
  \draw[gray!65] (1.2,-2.15)--(1.2,2.15)--(11.75,1.8)--(11.75,-1.8)--cycle;
  \draw[very thick] (1.2,-2.45)--(1.2,2.45);
  \node[left] at (1.1,2.36) {$L$};
  \node at (11.2,2.08) {$\Pi_i$};

  \coordinate (xx) at (3.8,.7);
  \coordinate (coreone) at (1.2,1.814);
  \coordinate (coretwo) at (1.2,-.414);
  \coordinate (petalone) at (1.2,.44);
  \coordinate (petaltwo) at (1.2,1.23);

  \draw[thick,gray!75] (coreone)--(4.85,.25);
  \draw[thick,gray!75] (coretwo)--(4.85,1.15);
  \draw[densely dashed,gray!65] (petalone)--(7.3,1.05);
  \draw[densely dashed,gray!65] (petaltwo)--(7.75,-.1);
  \node[gray!70] at (5.05,1.72) {core directions};
  \node[gray!70] at (6.0,-1.48) {private directions};

  \fill (coreone) circle (2.2pt) node[left] {$k_1$};
  \fill (coretwo) circle (2.2pt) node[left] {$k_2$};
  \draw[fill=white] (petalone) circle (2.2pt) node[left] {$p_y$};
  \draw[fill=white] (petaltwo) circle (2.2pt) node[left] {$p_{y'}$};

  \draw[fill=white,draw=gray!75] (3.8,.7)--(4.15,1.15)--(4.85,1.15)--(5.2,.7)--(4.85,.25)--(4.15,.25)--cycle;
  \foreach \p in {(3.8,.7),(4.15,1.15),(4.85,1.15),(5.2,.7),(4.85,.25),(4.15,.25)} \fill \p circle (1.45pt);
  \fill (xx) circle (2.1pt) node[above left] {$x$};
  \node at (4.5,-.08) {$Q_1$};

  \draw[fill=white,draw=gray!75] (7.3,1.05)--(7.65,1.5)--(8.35,1.5)--(8.7,1.05)--(8.35,.6)--(7.65,.6)--cycle;
  \foreach \p in {(7.3,1.05),(7.65,1.5),(8.35,1.5),(8.7,1.05),(8.35,.6),(7.65,.6)} \fill \p circle (1.45pt);
  \node[right] at (8.82,1.42) {$Q_2$};

  \node at (8.05,.25) {$\vdots$};
  \draw[fill=white,draw=gray!75] (7.4,-.55)--(7.75,-.1)--(8.45,-.1)--(8.8,-.55)--(8.45,-1)--(7.75,-1)--cycle;
  \foreach \p in {(7.4,-.55),(7.75,-.1),(8.45,-.1),(8.8,-.55),(8.45,-1),(7.75,-1)} \fill \p circle (1.45pt);
  \node[right] at (8.82,-.78) {$Q_m$};

  \node at (-6.45,-2.78) {(a)};
  \node at (6.45,-2.78) {(b)};
\end{tikzpicture}
\caption{The free-core sunflower book of Lemma~\ref{lem:free-core-book}. (a) The graph $J$ and its independent set $X$. The neighborhoods $N_J(x)=K\sqcup P_x$, $x\in X$, form a sunflower with core $K$, and $X$ is partitioned into blocks $X_1,\ldots,X_q$. (b) The placement of one block $X_i$ in the page $\Pi_i$. The sets $Q_1,\ldots,Q_m$ are homothetic copies of a regular $2r$-gon $Q$, whose distinguished directions are indexed by the core vertices placed on the line at infinity $L$. Internal secants give every point of $X_i$ a special line in every core direction. Each point also has $(m-1)|Q|$ private cross-copy directions; after orienting the cross-copy secants as in the proof, at least $r(m-1)\geq|P_x|$ are assigned to $x$. Placing the vertices of $P_x$ at the corresponding points of $L$ makes these lines special.}
\label{fig:free-core-book}
\end{figure}

\begin{proof}
If $|X|<10D$, the collinear realization gives $\SGdim(J)\geq 1>|X|/(10D)$. We may therefore assume that $|X|\geq 10D$. Let $K$ be the core and write
\begin{equation}
\label{eq:free-core-sunflower}
N_J(x)=K\sqcup P_x
\qquad(x\in X).
\end{equation}
The petals are pairwise disjoint, and independence gives $X\cap(K\cup P_x)=\emptyset$. By the degree bound,
\begin{equation}
\label{eq:free-core-petal-bound}
|K|\leq D,\qquad
|P_x|\leq D-|K|.
\end{equation}
Set
\begin{equation}
\label{eq:free-core-parameters}
a=|K|,\qquad r=\max\{1,a\},\qquad d=D-a,
\qquad m=1+\left\lceil\frac dr\right\rceil,\qquad b=2rm.
\end{equation}
Using~\eqref{eq:free-core-parameters}, if $a=0$, then $b=2(D+1)\leq 4D$. If $a\geq 1$, then
\[
b\leq 2a\left(2+\frac{D-a}{a}\right)=2D+2a\leq 4D.
\]
Thus, in both cases we have
\begin{equation}
\label{eq:free-core-block-size}
b\leq 4D.
\end{equation}
Put
\begin{equation}
\label{eq:free-core-page-count}
q=\left\lfloor\frac{|X|-3}{b}\right\rfloor.
\end{equation}
Then $q\geq 1$. Choose disjoint blocks
\[
X_1,\ldots,X_q,
\qquad |X_j|=b.
\]
At least three vertices of $J$ remain outside their union.

Take a $q$-page book with spine $L$ and apply Lemma~\ref{lem:polygon-directions} with $k=r$ to obtain a set $Z$ of $2r$ affine points and a set $\mathcal D$ of $r$ directions. Identify the line at infinity of this affine model with $L$, so that the directions in $\mathcal D$ become points of $L$. Assign the points of $\mathcal D$ bijectively to $K$ (if $a=0$ skip this step) and place $k\in K$ at the corresponding point $q_k\in L$.

For each $j$, partition $X_j$ into $m$ parts of size $2r$, each identified with a copy of $Z$. Suppose first that $m \geq 2$ and  let $H_j$ be the complete $m$-partite graph on the $m$ parts comprising $X_j$. Every vertex has even degree $2r(m-1)$. Orient an Euler tour in each nontrivial component of $H_j$; by~\eqref{eq:free-core-parameters}, every vertex then has outdegree
\[
r(m-1)\geq d.
\]
By~\eqref{eq:free-core-petal-bound} and~\eqref{eq:free-core-parameters}, we may assign $P_x$ injectively to the outgoing edges at $x$. Since each edge has only one tail, no edge receives two assignments.  If $m=1$, then $d=0$ and there is nothing to assign.

\begin{claim}\label{claim:free-core-page}
Let $F\subset L$ be finite and contain the points $q_k$, $k\in K$. A block $X_j$, together with the petals $P_x$, $x\in X_j$, can be placed in a page so that all points of $X_j$ lie off $L$, all points of $P_x$ lie in $L\setminus F$ for every $x\in X_j$, and every edge connected to  $X_j$  is special.
\end{claim}
\begin{proof}
Apply Lemma~\ref{lem:generic-homothetic} to get $m$ homothetic copies of $Z$ satisfying the secant conditions of the lemma (all cross-copy secants have distinct directions). Use the resulting copies of $Z$ to place the $m$ parts of $X_j$. Internal chords in every core direction pair the points within each copy. For $y\in P_x$, place $y$ at the point at infinity of the secant assigned to it in the previous step. The assigned secants have distinct directions outside $F$, so these placements are distinct and realize all required incidences.
\end{proof}

We now construct the pages in order, applying Claim~\ref{claim:free-core-page} in page $\Pi_j$ with $F$ consisting of the core points and all points of $P_x$ for $x\in X_1\cup\cdots\cup X_{j-1}$. Thus all page configurations are mutually disjoint except for their common core on $L$.

Place every remaining vertex of $J$ at a distinct unused point of $L$. At least three vertices lie on $L$. An edge with no endpoint in a block $X_i$ is therefore special on $L$. If $x\in X_j$ and $xy\in E(J)$, then~\eqref{eq:free-core-sunflower} gives $y\in K\sqcup P_x$, and the page construction supplies a third point on the line through $x$ and $y$. There are no edges between vertices in $\bigcup_iX_i$ because $X$ is independent. Hence the configuration realizes $J$.

Every page contains at least one point off $L$, so the configuration spans the entire $q$-page book, namely $\RP^{q+1}$. Passing to an affine chart gives
\[
\SGdim(J)\geq q+1.
\]
Finally, using~\eqref{eq:free-core-block-size} and~\eqref{eq:free-core-page-count},
\[
q+1>
\frac{|X|-3}{b}
\geq\frac{|X|-3}{4D}
\geq\frac{|X|}{10D}.
\]
This completes the proof of Lemma~\ref{lem:free-core-book}.
\end{proof}

\begin{obs}[The complete bipartite graph] 
Let $m \geq s$ and let $G = K_{m,s}$ be the complete bipartite graph on vertices $[m] \times [s]$. A direct application of Lemma~\ref{lem:free-core-book} gives $\SGdim(G) \geq \Omega(m/s)$ whereas Theorem~\ref{thm:quantitative-sg} gives an asymptotically matching upper bound.
\end{obs}

\subsection{Sunflowers with different cores}

For the proof of the  lower bound for minor-free graphs (Theorem~\ref{thm:main-fixed-minor}) we will need to handle a situation when the neighborhoods of the independent set form many different sunflowers with possibly intersecting cores (petals will still be disjoint, even between different sunflowers). To place these on pages, as before, we need a construction in which the core vertices are already placed on the spine in advance (since they may belong to multiple sunflowers). Using the polygon we can get $k$ secant directions (which we are free to choose) through each point using only $2k$ points. To get arbitrary $k$ secant directions through each point we have to use $2^k$ points placed on a projected hypercube. The next lemma describes how this is done in a single page and will be iterated in Lemma~\ref{lem:iterated-fixed-core-book} to get the full book construction. Figure~\ref{fig:fixed-core-page} illustrates the construction.

\begin{lem}[Fixed-core sunflower page]\label{lem:fixed-core-page}
Let $D\geq 1$, $a \geq 0$ be integers, let $J=(V,E)$ be a graph with $I \subset V$ an independent set satisfying:

\[ |V\setminus I|\geq 3, \qquad 
d_J(x)\leq D\,\,\,\,\, \forall x \in I \qquad \text{and} \qquad
|I|=2^a(2D+1).
\]
Suppose that the neighborhoods $\{ N_J(x)$, $x\in I\}$, form a sunflower with core $K$ such that $a = |K|$.
Let $\Pi$ be a real projective plane, let $L\subset\Pi$ be a line, let $A\subset L$ satisfy $|A|=|K|$, and let $\phi\colon K\to A$ be a bijection. Let $F\subset L\setminus A$ be an arbitrary finite set. Then, there is an injective map $
f\colon V\longrightarrow\Pi
$
such that
\begin{enumerate}
\item $f|_K=\phi$;
\item $f(X)\subseteq\Pi\setminus L$ and $f(V\setminus I)\subseteq L$;
\item for every $uv\in E$, the line through $f(u)$ and $f(v)$ contains a third point of $f(V)$;
\item $f(V)\cap F=\emptyset$.
\end{enumerate}
\end{lem}

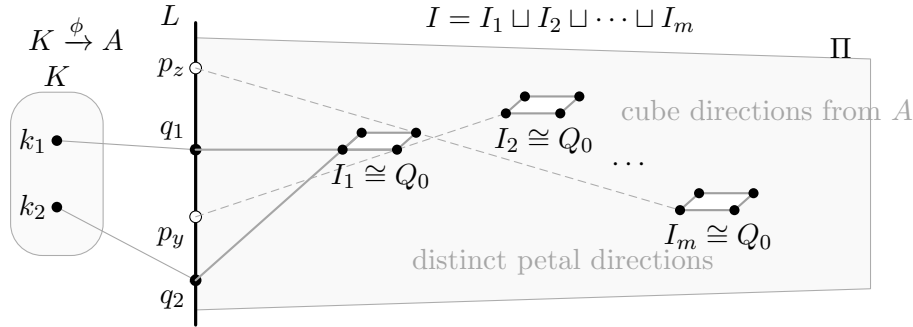
\begin{figure}[h!]
\centering
\begin{tikzpicture}[x=.72cm,y=.80cm,line cap=round,line join=round]
  \fill[gray!5] (1.5,-2.25)--(1.5,2.25)--(13.9,1.9)--(13.9,-1.9)--cycle;
  \draw[gray!65] (1.5,-2.25)--(1.5,2.25)--(13.9,1.9)--(13.9,-1.9)--cycle;
  \draw[very thick] (1.5,-2.5)--(1.5,2.5);
  \node[left] at (1.4,2.62) {$L$};
  \node at (13.35,2.08) {$\Pi$};
  \node at (8.2,2.55) {$I=I_1\sqcup I_2\sqcup\cdots\sqcup I_m$};

  \coordinate (kone) at (-1.05,.55);
  \coordinate (ktwo) at (-1.05,-.55);
  \coordinate (qone) at (1.5,.4);
  \coordinate (qtwo) at (1.5,-1.76);
  \coordinate (py) at (1.5,-.71);
  \coordinate (pz) at (1.5,1.75);
  \draw[rounded corners=12pt,fill=gray!4,draw=gray!65] (-1.9,-1.35) rectangle (-.2,1.35);
  \node at (-1.05,1.62) {$K$};
  \draw[->,gray!70] (kone)--(qone);
  \draw[->,gray!70] (ktwo)--(qtwo);
  \node at (-.7,2.35) {$K\xrightarrow{\phi}A$};
  \fill (kone) circle (2pt) node[left] {$k_1$};
  \fill (ktwo) circle (2pt) node[left] {$k_2$};
  \fill (qone) circle (2.2pt) node[above left] {$q_1$};
  \fill (qtwo) circle (2.2pt) node[below left] {$q_2$};
  \draw[fill=white] (py) circle (2.2pt) node[below left] {$p_y$};
  \draw[fill=white] (pz) circle (2.2pt) node[left] {$p_z$};

  \coordinate (aone) at (4.2,.4);
  \coordinate (bone) at (5.2,.4);
  \coordinate (cone) at (4.55,.68);
  \coordinate (done) at (5.55,.68);
  \coordinate (atwo) at (7.2,1.0);
  \coordinate (btwo) at (8.2,1.0);
  \coordinate (ctwo) at (7.55,1.28);
  \coordinate (dtwo) at (8.55,1.28);
  \coordinate (am) at (10.4,-.6);
  \coordinate (bm) at (11.4,-.6);
  \coordinate (cm) at (10.75,-.32);
  \coordinate (dm) at (11.75,-.32);
  \draw[thick,fill=white,draw=gray!75] (aone)--(bone)--(done)--(cone)--cycle;
  \draw[thick,fill=white,draw=gray!75] (atwo)--(btwo)--(dtwo)--(ctwo)--cycle;
  \draw[thick,fill=white,draw=gray!75] (am)--(bm)--(dm)--(cm)--cycle;

  \draw[thick,gray!70] (qone)--(bone);
  \draw[thick,gray!70] (qtwo)--(cone);
  \draw[densely dashed,gray!65] (py)--(atwo);
  \draw[densely dashed,gray!65] (pz)--(am);
  \node[gray!70] at (12.0,1.05) {cube directions from $A$};
  \node[gray!70] at (8.25,-1.45) {distinct petal directions};

  \foreach \v in {aone,bone,cone,done,atwo,btwo,ctwo,dtwo,am,bm,cm,dm} \fill (\v) circle (1.9pt);
  \node at (4.9,-.05) {$I_1\cong Q_0$};
  \node at (7.9,.55) {$I_2\cong Q_0$};
  \node at (9.5,.15) {$\cdots$};
  \node at (11.1,-1.05) {$I_m\cong Q_0$};
\end{tikzpicture}
\caption{A diagram showing the construction in Lemma~\ref{lem:fixed-core-page}. The core vertices $K$ are mapped by $\phi$ to an arbitrary set $A\subset L$. Each block $I_i$ is placed on a copy of the cube $Q_0$ whose directions are determined by $A$. Cross-copy edges give distinct directions corresponding to the petals.}
\label{fig:fixed-core-page}
\end{figure}

\begin{proof}
Treat $L$ as the line at infinity of $\Pi$. Write
\[
K=\{k_1,\ldots,k_a\},
\qquad
A=\{q_1,\ldots,q_a\},
\qquad
\phi(k_i)=q_i.
\]
For each $i\in[a]$, let $u_i\in\R^2$ be the unit vector in direction $q_i$. Let
\[
P_x=N_J(x)\setminus K\qquad(x\in I).
\]
The petals $P_x$ are pairwise disjoint and satisfy
\begin{equation}
\label{eq:fixed-core-petal-bound}
|P_x|=d_J(x)-a\leq D-a\leq D.
\end{equation}

\begin{claim}\label{claim:fixed-core-cube}
There is a set $Q_0\subset\R^2$ of $2^a$ points such that every point of $Q_0$ belongs to a secant (of $Q_0$) in each of the directions $q_1,\ldots,q_a$.
\end{claim}

\begin{proof}
For $a=0$, take $Q_0=\{0\}$. Otherwise choose nonzero scalars $\lambda_1,\ldots,\lambda_a$ so that the points
\[
v_\varepsilon
=\sum_{i=1}^a\varepsilon_i\lambda_i u_i,
\qquad \varepsilon\in\{0,1\}^a,
\]
are distinct. The excluded choices lie in finitely many proper linear subspaces of $\R^a$. Flipping the $i$-th coordinate changes $v_\varepsilon$ by $\pm\lambda_i u_i$, and hence pairs the cube points in direction $q_i$. Take $Q_0=\{v_\varepsilon:\varepsilon\in\{0,1\}^a\}$.
\end{proof}

Put
\[
s=|Q_0|=2^a,
\qquad m=2D+1.
\]
Since $|I|=2^a(2D+1)=ms$, we can partition $I$ into $m$ parts $I_1,\ldots,I_m$ of size $s$, each identified with a copy of $Q_0$. Let $H$ be the complete $m$-partite graph on these $m$ copies. Every vertex has even degree
\[
s(m-1)=2sD.
\]
 Orient an Euler tour of $H$; every vertex then has outdegree $sD\geq D$ and, using ~\eqref{eq:fixed-core-petal-bound} can assign $P_x$ injectively to the outgoing edges at $x$. No edge receives two assignments.

Apply Lemma~\ref{lem:generic-homothetic} with $
Q=Q_0$ and $\cal F=A\cup F$ and use the resulting homothetic copies to place the $m$ parts $I_1,\ldots,I_m$ of $I$. If an outgoing edge $xz$ is assigned to $y\in P_x$, place $y$ at the point at infinity $xz\cap L$. Lemma~\ref{lem:generic-homothetic} guarantees that all points of $P_x$, $x\in I$, are distinct and avoid $A\cup F$. The cube pairings give, for every $x\in I$ and $k_i\in K$, a second point of $I$ on the line through $x$ and $q_i$.

Place every remaining vertex at a distinct unused point of $L\setminus(A\cup F)$. This defines an injective map $f$ avoiding $F$. Since $|V\setminus I|\geq 3$, every edge with both endpoints outside $I$ is special on $L$. If $xy\in E(J)$ with $x\in I$, then
\[
y\in N_J(x)=K\sqcup P_x,
\]
and the construction supplies a third point on the line through $f(x)$ and $f(y)$. There are no edges inside $I$. Thus every edge is special. This completes the proof of Lemma~\ref{lem:fixed-core-page}.
\end{proof}

The next lemma iterates Lemma~\ref{lem:fixed-core-page} and is used in the proof of Theorem~\ref{thm:main-fixed-minor}. Its input is an independent set $X$ partitioned into $p$ blocks whose neighborhoods form sunflowers. The cores may intersect, but all petals are pairwise disjoint and avoid every core. We place the union of the cores on the spine and construct one page for each block.

\begin{lem}[Intersecting core book]
\label{lem:iterated-fixed-core-book}
Let $D\geq 0$ and $p\geq 1$ be integers, let $J=(V,E)$ be a graph, and let $X=X_1\sqcup\cdots\sqcup X_p\subseteq V$ be an independent set partitioned into $p$ blocks. For every $i\in[p]$, suppose that the neighborhoods $N_J(x)$, $x\in X_i$, form a sunflower with core $K_i$, and for $x\in X_i$ write $P_x=N_J(x)\setminus K_i$. Suppose that
\[
d_J(x)\leq D\qquad(x\in X),
\qquad
|X_i|\geq 2^D(2D+1)+3\qquad(i\in[p]),
\]
and that the sets $P_x$, $x\in X$, are pairwise disjoint and avoid every core $K_i$. Then
\[
\SGdim(J)\geq p+1.
\]
\end{lem}

\begin{proof}
For a graph $H$ and $A\subseteq V(H)$, we write $H[A]$ for the subgraph of $H$ induced by $A$. Put
\begin{equation}
\label{eq:iterated-book-sets}
K^\ast=\bigcup_{i=1}^pK_i,
\qquad
W_i=X_i\cup\bigcup_{x\in X_i}P_x,
\qquad
J_i=J[W_i\cup K_i]
\quad(i\in[p]).
\end{equation}
Since $X$ is independent, it is disjoint from $K^\ast$ and from every petal. By~\eqref{eq:iterated-book-sets}, the sets $W_i$ are pairwise disjoint and avoid $K^\ast$, so $V(J_i)\cap V(J_j)=K_i\cap K_j$ whenever $i\neq j$.

Put $a_i=|K_i|$. Since $a_i\leq D$, we may choose
\[
Y_i\subseteq X_i,
\qquad
|Y_i|=2^{a_i}(2D+1).
\]
Our assumption gives $|X_i\setminus Y_i|\geq 3$. For every $x\in Y_i$,
\[
N_{J_i}(x)=N_J(x)=K_i\sqcup P_x.
\]
Thus the neighborhoods of $Y_i$ form a sunflower in $J_i$, every $x\in Y_i$ has degree at most $D$, and $|V(J_i)\setminus Y_i|\geq 3$.

Take a $p$-page book with spine $L$ and pages $\Pi_1,\ldots,\Pi_p$. Place the vertices of $K^\ast$ at distinct points of $L$, and construct the pages in order. Before constructing $\Pi_i$, let $A_i\subset L$ be the points at which we placed the vertices of $K_i$, and let
\[
\phi_i\colon K_i\longrightarrow A_i
\]
be the prescribed bijection. Let $F_i$ consist of the occupied points of $L$ outside $A_i$. Apply Lemma~\ref{lem:fixed-core-page} with
\[
\begin{aligned}
J&=J_i,\quad K=K_i,\quad I=Y_i,\end{aligned}
\]
to obtain an injective map
\[
f_i\colon V(J_i)\longrightarrow\Pi_i
\]
which agrees with the prescribed positions on $K_i$, places only $Y_i$ off $L$, avoids the earlier points, and realizes $J_i$. Since $Y_i$ is nonempty and $|V(J_i)\setminus Y_i|\geq 3$, its image spans $\Pi_i$.

These maps are compatible: $J_i$ and $J_j$ meet only in $K_i\cap K_j$, where both use the prescribed placement. All other spine points are fresh, and points outside $L$ lie in distinct pages.

Place every remaining vertex of $J$ at a distinct unused point of $L$. Each set $X_i\setminus Y_i$ supplies three spine points. By~\eqref{eq:iterated-book-sets}, if an edge has an endpoint $x\in Y_i$, its other endpoint belongs to $V(J_i)$, so the realization in $\Pi_i$ makes it special. Every other edge lies on $L$ and is special there.

The configuration spans every page, hence the entire book $\RP^{p+1}$. Passing to an affine chart proves $\SGdim(J)\geq p+1$.
\end{proof}

\section{Sparse graphs}
\label{sec:sparse}

In this section we prove our two lower bounds on the Sylvester--Gallai dimension of graphs with bounded maximum degree and bounded average degree. The case of bounded maximum degree is much easier and so we prove it first.

\begin{boundedmaximumdegreerestatement}[Bounded maximum degree]
Let $G$ be a graph on $n\geq 3$ vertices with maximum degree at most $D$, where $D\geq 1$. Then
\[
\SGdim(G)=\Omega\left(\frac{n}{D^3}\right).
\]
\end{boundedmaximumdegreerestatement}

\begin{proof}

Let $G^2$ denote the square of $G$: its vertex set is $V(G)$, with two vertices adjacent whenever their distance in $G$ is at most two. Then $G^2$ has maximum degree at most $D^2$ and, by greedy selection, has an independent set $X\subseteq V(G)$ with
\[
|X|\geq\frac{n}{D^2+1}.
\]
A common neighbor of distinct $x,z\in X$ would put them at distance two, so the neighborhoods $N_G(x)$, $x\in X$, are pairwise disjoint and form a sunflower with empty core. Thus, we can apply Lemma~\ref{lem:free-core-book} to obtain
\[
\SGdim(G)>\frac{|X|}{10D}
\geq\frac{n}{10D(D^2+1)}
\geq\frac{n}{20D^3}.
\]
\end{proof}

We next restate the harder bounded average-degree result, which will be proved in the rest of this section.

\begin{boundedaveragedegreerestatement}[Bounded average degree]
For every integer $D\geq 1$, every graph $G$ on $n\geq 3$ vertices with average degree at most $D$ satisfies
\[
\SGdim(G)=\Omega_D(n^{1/D}).
\]
\end{boundedaveragedegreerestatement}

The next simple lemma supplies the initial bounded-degree independent set in the proof of Theorem~\ref{thm:main-sparse}.

\begin{lem}[Low-degree independent set]\label{lem:low-degree-independent}
Let $D\geq 1$ be an integer. If $G$ has $n$ vertices and average degree at most $D$, then $G$ contains an independent set $X$ such that
\[
d_G(x)\leq D\quad(x\in X),
\qquad |X|\geq\frac{n}{(D+1)^2}.
\]
\end{lem}

\begin{proof}
Let $H$ be the set of vertices of degree at least $D+1$. The degree sum gives
\[
(D+1)|H|
\leq\sum_{v\in V(G)}d_G(v)
\leq Dn.
\]
Thus $U=V(G)\setminus H$ has at least $n/(D+1)$ vertices and $G[U]$ has maximum degree at most $D$. Greedy selection in $G[U]$ removes at most $D+1$ vertices at each step, and hence produces an independent set of size at least $n/(D+1)^2$.
\end{proof}

We will need the classical sunflower lemma to handle distinct neighborhood types in Theorem~\ref{thm:main-sparse}. One could use the more recent improvements of ~\cite{ALWZ,BellChueluechaWarnke}, but the gain would only affect the implicit constant in the final bound.
\begin{lem}[Sunflower lemma~\cite{ErdosRado}]\label{lem:sunflower}
Let $D\geq 1$ and $k\geq 2$. If $\mathcal A$ is a family of more than $D!(k-1)^D$ distinct sets, each of size at most $D$, then $\mathcal A$ contains a sunflower with $k$ petals.
\end{lem}

We will also need to apply the sunflower lemma for set systems in which some sets might repeat.

\begin{lem}[Sunflower lemma with repetitions]\label{lem:indexed-sunflower}
Let $D\geq 1$ and $k\geq 2$, and let $(F_i)_{i=1}^M$ be an indexed family of sets with $|F_i|\leq D$. If
\[
M>D!(k-1)^{D+1},
\]
then some $k$ indexed members form a sunflower.
\end{lem}

\begin{proof}
If some set occurs at least $k$ times, these occurrences form a sunflower. Otherwise every set occurs at most $k-1$ times, so the family contains more than $D!(k-1)^D$ distinct sets. Apply Lemma~\ref{lem:sunflower} with rank parameter $D$ and sunflower size $k$.
\end{proof}

\begin{proof}[{\bf Proof of Theorem~\ref{thm:main-sparse}}:]

Apply Lemma~\ref{lem:low-degree-independent} to $G$ to get an independent set $X_0$ whose vertices have degree at most $D$. Put
\begin{equation}
\label{eq:sparse-parameters}
M=|X_0|\geq\frac{n}{(D+1)^2},
\qquad
R=2^D(2D+1)+3
\end{equation}
and partition $X_0$ into classes of vertices so that $x$ and $y$ are in the same class if and only if $N_G(x)=N_G(y)$.

\emph{Case 1: Repeated neighborhood types.} Suppose that at least $M/2$ vertices lie in classes of size at least $R$. Split each such class into blocks of exactly $R$ vertices, discarding the remainders. If a class has size $z\geq R$, it contributes at least $z/(2R)$ blocks, so the total number $p$ of blocks satisfies
\[
p\geq\frac{M}{4R}.
\]
Within each block all neighborhoods are equal, and hence form a sunflower with empty petals. Lemma~\ref{lem:iterated-fixed-core-book} then gives
\[
\SGdim(G)\geq p+1\geq\frac{M}{4R}\geq\frac{n}{4(D+1)^2R} \geq \Omega_D(n^{1/D}).
\]
Here the last inequality uses~\eqref{eq:sparse-parameters}.

\emph{Case 2: Distinct neighborhood types.} Otherwise, more than $M/2$ vertices lie in classes of size less than $R$. The number $s$ of distinct neighborhood types among these vertices therefore satisfies
\begin{equation}
\label{eq:sparse-distinct-types}
s>\frac{M}{2R}.
\end{equation}
Put $k=\lceil(s/D!)^{1/D}\rceil$. If $k=1$, choose one neighborhood type. If $k\geq 2$, then $D!(k-1)^D<s$, so Lemma~\ref{lem:sunflower} gives $k$ distinct neighborhood types that form a sunflower. In either case, choose one representative of each selected type to obtain a set $X$ of size $k$ whose neighborhoods form a sunflower. Together with~\eqref{eq:sparse-parameters} and~\eqref{eq:sparse-distinct-types}, this gives
\[
k\geq\left(\frac{s}{D!}\right)^{1/D}=\Omega_D(n^{1/D}).
\]
Lemma~\ref{lem:free-core-book} now gives
\[
\SGdim(G)>\frac{k}{10D}=\Omega_D(n^{1/D}).
\]
This completes the proof of Theorem~\ref{thm:main-sparse}. 
\end{proof}

\section{Graphs with an  excluded minor}
\label{sec:minor-free}

In this section we restate and prove the linear lower bound on the Sylvester--Gallai dimension for graphs excluding some fixed minor.

\begin{fixedexcludedminorrestatement}[Fixed excluded minor]
For every finite graph $H$, there is a constant $c_H>0$ such that every $H$-minor-free graph $G$ on $n\geq 3$ vertices satisfies
\[
\SGdim(G)\geq c_H n.
\]
\end{fixedexcludedminorrestatement}

We begin with a subsection containing the preliminaries we will need from the rich literature on graph minors. These include a linear edge count upper bound, a linear bound on the neighborhood complexity, and a partitioning lemma. We then prove the main technical lemma and the theorem in the following two subsections.

\subsection{Preliminaries}
\label{subsec:minor-free-preliminaries}

\begin{define}[Minors and minor models]
\label{def:graph-minor}
A graph $H$ is a \emph{minor} of a graph $G$, written $H\preccurlyeq G$, if $H$ can be obtained from a subgraph of $G$ by contracting edges. Equivalently, there are pairwise disjoint nonempty connected vertex sets $   (X_h)_{h\in V(H)}$
in $G$, called \emph{branch sets}, such that whenever $hh'\in E(H)$, some edge of $G$ joins $X_h$ to $X_{h'}$. A graph is $H$-minor-free if $H\not\preccurlyeq G$.
\end{define}

Notice that every subgraph of $G$ is a minor of $G$, and the minor relation is transitive. Thus, if $H\preccurlyeq K_t$, then every $H$-minor-free graph is $K_t$-minor-free. The first result we need is an upper bound on the number of edges in a graph not containing a $K_t$ minor.

\begin{thm}[Complete-minor density bound~\cite{Mader}]
\label{thm:complete-minor-density}
For every integer $t\geq 3$ there is a constant $\mu_t > 0$ such that every nonempty finite simple graph $G$ with $K_t\not\preccurlyeq G$ satisfies
\[
|E(G)|<\mu_t|V(G)|.
\]
\end{thm}


The second preliminary lemma controls how many different neighborhoods vertices may have inside a prescribed set. It can be derived from general results on bounded expansion: every proper minor-closed class has bounded expansion~\cite{NesetrilOssonaExpansion}, and bounded-expansion classes have linear neighborhood complexity~\cite{ReidlSanchezStavropoulos}. We include a self-contained proof following the argument of ~\cite[Section~3]{ReedWood}.

\begin{lem}[Linear neighborhood complexity]
\label{lem:minor-free-neighborhood-complexity}
For every integer $t\geq 3$, there is a constant $\nu_t\geq 1$ such that, for every $K_t$-minor-free graph $G$ and every $B\subseteq V(G)$,
\[
   \left|\{N_G(v)\cap B:v\in V(G)\setminus B\}\right|
   \leq\nu_t|B|+1.
\]
\end{lem}

\begin{proof}
Let
\[
\mathcal T=\{N_G(v)\cap B:v\in V(G)\setminus B\}.
\]
We refer to the elements of $\cT$ as `traces'. For each trace $S\in\mathcal T$, choose a representative $v_S\in V(G)\setminus B$ so that
\begin{equation}
\label{eq:trace-representative}
N_G(v_S)\cap B=S.
\end{equation}

If $B=\emptyset$, the result is immediate, so assume that $B\neq\emptyset$. Consider the bipartite graph with sides
\[
\mathcal T_{\geq 2}=\{S\in\mathcal T:|S|\geq 2\},\qquad \binom B2
\]
joining $S$ to each pair contained in $S$. We now describe a matching $M$ in this graph. Start by matching every two-element $S\in\cT$ to itself, and extend this to a maximal matching $M$. Let $Q$ be the graph on $B$ whose edges are the pairs in $\binom B2$ that are incident with an edge of $M$. Notice that maximality of $M$ implies that every unmatched $S$ is a clique of $Q$.

\begin{claim}\label{claim:neighborhood-cliques}
The graph $Q$ is a minor of $G$, and $\mathcal T$ injects into the cliques of $Q$, with the empty set counted as a clique.
\end{claim}

\begin{proof}
By~\eqref{eq:trace-representative}, if $S$ is matched to $\{x,y\}$, then $xv_S,v_Sy\in E(G)$. These paths have distinct internal vertices, and together with the isolated vertices of $Q$ they form a subdivision of $Q$. Hence $Q\preccurlyeq G$.

Map the empty and singleton traces to themselves, each trace matched by $M$ to its matched pair, and each unmatched trace $S$ to $S$ (which must be a clique by the preceding comment). Every two-element trace was matched at the outset; hence the four types of images have sizes zero, one, two, and at least three. Within each type the map is injective.
\end{proof}

It remains to count the cliques of $Q$. Let $\mu_t$ be as in Theorem~\ref{thm:complete-minor-density} and set $d_t=\lceil 2\mu_t\rceil-1$. By the above claim, every nonempty subgraph $Q'\subseteq Q$ is $K_t$-minor-free and so, by Theorem~\ref{thm:complete-minor-density}, we have
\[
|E(Q')|<\mu_t|V(Q')|.
\]
Thus every nonempty subgraph of $Q$ has a vertex of degree at most $d_t$. Repeatedly deleting such a vertex gives an ordering of $B$ in which each vertex has at most $d_t$ later neighbors. Every nonempty clique is determined by its first vertex in this ordering and a subset of its later neighbors. Therefore $Q$ has at most $2^{d_t}|B|$ nonempty cliques. Claim~\ref{claim:neighborhood-cliques} now gives $
|\mathcal T|\leq 2^{d_t}|B|+1,$ and 
so we may take $
\nu_t=2^{d_t}.$
This completes the proof of Lemma~\ref{lem:minor-free-neighborhood-complexity}.
\end{proof}

The last preliminary result we will need for graphs with an excluded minor will allow us to find a good division of the vertex set into not-too-overlapping parts. We first define the notion of graph division.

\begin{define}[Divisions and boundaries]
\label{def:graph-division}
A \emph{division} of a graph $G$ is a partition
\[
   E(G)=E_1\sqcup\cdots\sqcup E_p.
\]
The corresponding \emph{regions} $R_i$ are the edge-induced subgraphs with edge sets $E_i$; thus $V(R_i)$ consists of the endpoints of the edges in $E_i$. A vertex is a \emph{boundary vertex} if it belongs to more than one region, and the boundary of $R_i$ is
\[
   \partial R_i
   =
   \{v\in V(R_i):v\in V(R_j)\text{ for some }j\neq i\}.
\]
\end{define}

The next lemma is the special case of Yuster~\cite[Lemma~2.3]{Yuster} and follows from an iterative application of the weighted separator theorem of Alon, Seymour, and Thomas~\cite[Theorem~1.2]{AlonSeymourThomas} (see also Frederickson~\cite[Section~2]{Frederickson}).

\begin{lem}[Minor-free $r$-divisions]
\label{lem:minor-free-r-division}
For every integer $t\geq 3$, there are constants $\alpha_t,\beta_t\geq 1$
and an integer $r_t\geq 2$ such that the following holds. Let $r_t\leq r\leq n$ be integers. Every connected $n$-vertex $K_t$-minor-free graph $G$ has a division into regions $R_1,\ldots,R_p$ satisfying
\[
   p\leq\alpha_t\frac nr,\qquad
   |V(R_i)|\leq r,\qquad
   |\partial R_i|\leq\beta_t\sqrt r.
\]
\end{lem}

\subsection{Sunflower extraction for minor-free graphs}

The next lemma combines minor-free density (Theorem~\ref{thm:complete-minor-density}), linear neighborhood complexity (Lemma~\ref{lem:minor-free-neighborhood-complexity}), and the sunflower lemma with repetitions (Lemma~\ref{lem:indexed-sunflower}) to find a sunflower in a sufficiently large region of a graph with an excluded $K_t$. In the proof of Theorem~\ref{thm:main-fixed-minor}, we will use it iteratively to extract sunflowers from each region in a division of $G$. In the lemma statement, one should think of $U$ as the interior of the region and $B$ as the boundary.

For $t\geq 3$, recall the constant $\mu_t$ from Theorem~\ref{thm:complete-minor-density} and fix 
$$\Delta_t=\lceil 8\mu_t\rceil. $$

\begin{lem}[Sunflower extraction in a large region]
\label{lem:minor-free-region-sunflower}
For every integer $t\geq 3$ and every real $A\geq 1$, there is an integer $M_{t,A}$ with the following property. Let $J$ be $K_t$-minor-free with
\[
   V(J)=U\sqcup B,\qquad
   |U|=m\geq M_{t,A},\qquad
   |B|\leq A\sqrt m.
\]
Then there is an independent set $X\subseteq U$ such that $d_J(x)\leq\Delta_t$ for every $x\in X$, the neighborhoods $N_J(x)$, $x\in X$, form a sunflower whose petals lie in $U$, and
\[
|X|\geq 2^{\Delta_t}(2\Delta_t+1)+3.
\]
\end{lem}

\begin{proof}
Let $\mu=\mu_t$, and let $\nu_t$ be given by Lemma~\ref{lem:minor-free-neighborhood-complexity}. We first prove:

\begin{claim}\label{claim:minor-free-low-degree-trace}
There is a constant $c_{t,A} > 0$ and an independent set $I_0\subseteq U$ of size at least $c_{t,A}\sqrt m$ such that every vertex of $I_0$ has degree at most $\Delta_t$ and all vertices of $I_0$ have the same neighborhood in $B$.
\end{claim}

\begin{proof}
Assume that $M_{t,A}$ is large enough that $|B|\leq A\sqrt m\leq m$. By Theorem~\ref{thm:complete-minor-density},
\[
|E(J)|<\mu|V(J)|=\mu(m+|B|) \leq 2\mu m.
\]
Hence
\[
\sum_{u\in U}d_J(u)<4\mu m.
\]
Since $\Delta_t\geq 8\mu$, at least $m/2$ vertices of $U$ have degree at most $\Delta_t$. Greedy selection among them gives an independent set $I\subseteq U$ with
\[
|I|\geq\frac{m}{2(\Delta_t+1)}.
\]

Applying Lemma~\ref{lem:minor-free-neighborhood-complexity} with the boundary set $B$ gives
\[
\left|\{N_J(v)\cap B:v\in V(J)\setminus B\}\right|
\leq\nu_t|B|+1
\leq(\nu_t A+1)\sqrt m.
\]
Hence, partitioning the vertices $x\in I$ by the intersections $N_J(x)\cap B$ yields a class $I_0$ of size at least $c_{t,A}\sqrt m$, for a constant $c_{t,A}>0$.
\end{proof}

Fix $I_0$ as in the claim and let $K=N_J(x)\cap B$, for any $x\in I_0$, be its common neighborhood in $B$. Put
\[
k_0=2^{\Delta_t}(2\Delta_t+1)+3.
\]
Increasing $M_{t,A}$ if necessary, we may assume that
\[
c_{t,A}\sqrt m>\Delta_t!(k_0-1)^{\Delta_t+1}.
\]
For $x\in I_0$, write
\[
N_J(x) = F_x \sqcup K,
\]
so that $F_x\subseteq U$ and $|F_x|\leq\Delta_t$. Apply Lemma~\ref{lem:indexed-sunflower} to the indexed family $(F_x)_{x\in I_0}$ with rank parameter $\Delta_t$ and sunflower size $k_0$. It gives a set $X\subseteq I_0$ of size $k_0$ such that the sets $F_x$, $x\in X$, form a sunflower. If $T\subseteq U$ is its core and $P_x=F_x\setminus T$, then
\[
N_J(x)=(K\sqcup T)\sqcup P_x\qquad(x\in X),
\]
and the petals $P_x$ are pairwise disjoint subsets of $U$. This completes the proof of Lemma~\ref{lem:minor-free-region-sunflower}.
\end{proof}

\subsection{Proof of Theorem~\ref{thm:main-fixed-minor}}

Put $ t=|V(H)|$ and assume w.l.o.g that $t \geq 3$. Then $H\preccurlyeq K_t$, so every $H$-minor-free graph is $K_t$-minor-free. It is enough to prove the theorem when $G$ is connected. Indeed, let $C_1,\ldots,C_s$ be the components of $G$, choose $v_i\in V(C_i)$, and let $\widehat G$ be obtained by adding the `bridge' edges $v_iv_{i+1}$ for $i\in[s-1]$. It is easy to see that $\widehat G$ is still $K_t$-minor-free.
%

Let $\alpha_t,\beta_t,r_t$ be the constants given by Lemma~\ref{lem:minor-free-r-division}. Put
\begin{equation}
\label{eq:minor-free-parameters}
\eta=\frac{1}{4\alpha_t},
\qquad
A=\frac{\beta_t}{\sqrt\eta}.
\end{equation}
Choose an integer $r\geq r_t$, depending only on $t$, such that
\begin{equation}
\label{eq:minor-free-r-choice}
\frac{\alpha_t\beta_t}{\sqrt r}\leq\frac{1}{4},
\qquad
\eta r\geq M_{t,A},
\end{equation}
where $M_{t,A}$ is given by Lemma~\ref{lem:minor-free-region-sunflower}, and set $$c_H=1/(2r).$$
If $n<r$, then $\SGdim(G)\geq 1>c_H n$, so we may assume that $n\geq r$.

We now apply Lemma~\ref{lem:minor-free-r-division} with clique parameter $t$ and division parameter $r$. We obtain a division of $G$ into regions $R_1,\ldots,R_p$ satisfying
\begin{equation}
\label{eq:minor-free-division-bounds}
p\leq\alpha_t\frac nr,\qquad
|V(R_i)|\leq r,\qquad
|\partial R_i|\leq\beta_t\sqrt r.
\end{equation}
Put $Z=\bigcup_i\partial R_i$ and $U_i=V(R_i)\setminus Z$. The sets $U_i$ partition $V(G)\setminus Z$.

\begin{claim}\label{claim:many-large-interiors}
Let $q$ be the number of regions satisfying $|U_i|\geq\eta r$. Then
\begin{equation}
\label{eq:many-large-interiors}
q\geq\left\lceil\frac{n}{2r}\right\rceil.
\end{equation}
\end{claim}

\begin{proof}
By~\eqref{eq:minor-free-r-choice} and~\eqref{eq:minor-free-division-bounds}, the total boundary satisfies
\begin{equation}
\label{eq:minor-free-total-boundary}
|Z|
\leq\sum_i|\partial R_i|
\leq\frac{\alpha_t\beta_t}{\sqrt r}\,n
\leq\frac n4.
\end{equation}
The interiors with fewer than $\eta r$ vertices contain altogether at most
\[
p\eta r\leq n/4
\]
vertices, using~\eqref{eq:minor-free-parameters} and~\eqref{eq:minor-free-division-bounds}. Together with~\eqref{eq:minor-free-total-boundary}, this shows that the remaining interiors contain at least $n/2$ vertices. Each has at most $r$ vertices by~\eqref{eq:minor-free-division-bounds}, proving Claim~\ref{claim:many-large-interiors}.
\end{proof}

Relabel the heavy interiors as $U_1,\ldots,U_q$. Every edge incident to $U_i$ belongs to the region $R_i$, and therefore
\begin{equation}
\label{eq:minor-free-boundary-neighborhood}
B_i:=N_G(U_i)\setminus U_i\subseteq\partial R_i\subseteq Z.
\end{equation}
Using~\eqref{eq:minor-free-parameters},~\eqref{eq:minor-free-r-choice}, and~\eqref{eq:minor-free-division-bounds} gives
\[
|B_i|\leq\beta_t\sqrt r\leq A\sqrt{|U_i|},
\qquad
|U_i|\geq\eta r\geq M_{t,A}.
\]

Together with~\eqref{eq:minor-free-boundary-neighborhood}, these bounds allow us to apply Lemma~\ref{lem:minor-free-region-sunflower} on each $J_i=G[U_i\sqcup B_i]$ to obtain an independent set $X_i\subseteq U_i$ whose neighborhoods form a sunflower with petals in $U_i$. Let $K_i$ be the core of this sunflower and, for $x\in X_i$, put
\[
P_x=N_{J_i}(x)\setminus K_i\subseteq U_i.
\]

By~\eqref{eq:minor-free-boundary-neighborhood}, we have $N_{J_i}(x)=N_G(x)$ for $x\in X_i$. Hence the same core and petals apply in $G$, and
\begin{equation}
\label{eq:minor-free-block-bounds}
d_G(x)\leq\Delta_t\quad(x\in X_i),
\qquad
|X_i|\geq 2^{\Delta_t}(2\Delta_t+1)+3.
\end{equation}

By~\eqref{eq:minor-free-boundary-neighborhood}, no edge joins two distinct interiors $U_i,U_j$: every neighbor of $U_i$ outside $U_i$ lies in $B_i\subseteq Z$, whereas $U_j\cap Z=\emptyset$. Hence
\[
X=X_1\sqcup\cdots\sqcup X_q
\]
is independent. The petals are pairwise disjoint: this holds within each block by the sunflower property, while petals from different blocks lie in disjoint interiors. Moreover, $K_i\subseteq U_i\cup B_i$. Thus a petal from block $i$ avoids $K_i$ by the sunflower property and avoids every $K_j$, $j\neq i$, because $U_i$ is disjoint from $U_j\cup B_j$.

By~\eqref{eq:minor-free-block-bounds} and the preceding observations, we may apply Lemma~\ref{lem:iterated-fixed-core-book} (intersecting cores book) on $G$ and $X=X_1\sqcup\cdots\sqcup X_q$. Together with~\eqref{eq:many-large-interiors}, the lemma gives
\[
\SGdim(G)\geq q+1\geq\frac{n}{2r}=c_H n.
\]
This completes the proof of Theorem~\ref{thm:main-fixed-minor}.
\qed

\section{Random graphs}
\label{sec:random-out}

Recall the symmetrized random $D$-out graph $\widehat G_{n,D}$ from Definition~\ref{def:symmetrized-random-out}. We restate the main theorem proven in this section.

\begin{randomoutrestatement}[Random graph bound]
There are absolute constants $A>0$ and $K>1$ such that, for all sufficiently large $n$ and all integers $K\log n\leq D\leq n-1$,
\[
\SGdim(\widehat G_{n,D})
\leq
\max\left\{2,\left\lceil\dfrac{An\log n}{D^2\log(D/\log n)}\right\rceil\right\}
\]
with probability $1-o(1)$ as $n\to\infty$.
\end{randomoutrestatement}

We develop the necessary machinery in the next three subsections, followed by the proof of the Theorem.

\subsection{The number of ordinary graphs}

The union bound in the proof of Theorem~\ref{thm:main-random-tradeoff} will need the following bound on the number of graphs arising from ordinary lines in point configurations.

\begin{define}[Ordinary graphs]
\label{def:ordinary-graphs}
Let $P=(p_1,\ldots,p_n)$ be a labeled configuration of distinct points in $\R^d$. The \emph{ordinary graph} $\Ord(P)$ is the graph on $[n]$ in which $ij$ is an edge if and only if the line $p_ip_j$ contains no other point of $P$. For $r,n\geq 1$, let
\[
\Ord(n)_{\geq r}=\{\Ord(P):|P|=n\text{ and } \affdim(P) \geq r\}, \qquad \Ord(n) = \bigcup_r \Ord(n)_{\geq r}
\]
where the configurations may lie in any ambient dimension.
\end{define}

Thus a labeled configuration $P$ is a special-line realization of $G$ precisely when
\begin{equation}
\label{eq:ordinary-avoidance-criterion}
E(G)\cap E(\Ord(P))=\emptyset.
\end{equation}
Hence, to show that $\SGdim(G) < r$ we have to show that $G$ `hits' an edge in all graphs $O \in \Ord(n)_{\geq r}$. To upper bound the number of ordinary graphs we will need the following polynomial zero-pattern bound of R\'onyai, Babai, and Ganapathy.

\begin{prop}[Zero-patterns of low degree maps {\cite{RonyaiBabaiGanapathy}}]
\label{prop:zero-pattern-bound}
Let $M\geq N\geq 1$ and $r\geq 1$, and let $F_1,\ldots,F_M$ be real polynomials of degree at most $t$ in $N$ variables. Then the number of zero patterns
\[
\{i\in[M]:F_i(x)=0\},
\qquad x\in\R^N,
\]
is at most
\[
\binom{Mt-(t-2)N}{N}.
\]
\end{prop}

\begin{lem}[Number of ordinary graphs]
\label{lem:ordinary-graph-entropy}
For every $n\geq 1$,
\[
|\Ord(n)|\leq\exp(Bn\log n),
\]
where $B$ is an absolute constant.
\end{lem}

\begin{proof}
Assume w.l.og that $n\geq 5$ and apply a generic affine map to $\R^2$ preserving distinctness and the collinearity status of every triple. This preserves the ordinary graph, so it is enough to count configurations in $\R^2$.

Write $p_i=(p_{i,1},p_{i,2})$. For $1\leq i<j<k\leq n$, let
\[
F_{ijk}
=(p_{j,1}-p_{i,1})(p_{k,2}-p_{i,2})
-(p_{j,2}-p_{i,2})(p_{k,1}-p_{i,1}).
\]
The triple $p_i,p_j,p_k$ is collinear precisely when $F_{ijk}$ vanishes. Thus the zero pattern of these polynomials determines $\Ord(P)$. There are
$
M=\binom n3
$
quadratic polynomials in $N=2n$ variables and so Proposition~\ref{prop:zero-pattern-bound} gives
\[
|\Ord(n)|
\leq \binom{2M}{N}
=\exp(O(n\log n)).
\]
As was claimed.
\end{proof}

\subsection{A relative Sylvester--Gallai bound}

We use the following terminology from ~\cite{BDWY,DSW}. For a vector $u$, let $\supp(u)$ denote the set of its nonzero coordinates. An $m\times n$ matrix $A$  is a $(q,k,t)$-design matrix if every row has support of size at most $q$, every column has support of size at least $k$, and the supports of any two distinct columns intersect in at most $t$ coordinates. For $0<\delta\leq 1$, distinct points $v_1,\ldots,v_n\in\R^d$ form a $\delta$-SG configuration if, for every $i$, at least $\delta(n-1)$ of the other points lie on special lines through $v_i$. We will need the following result from~\cite{DSW} (strengthening an earlier bound of~\cite{BDWY}).

\begin{prop}[Design-matrix rank bound {\cite{DSW}}]
\label{prop:dsw-design-rank}
Let $A$ be a real or complex $m\times n$ $(q,k,t)$-design matrix. Then
\[
\operatorname{rank}(A)
\geq
\frac{n}{1+q(q-1)t/k}.
\]
\end{prop}

The next result of~\cite{DSW} supplies the linear dependencies that convert a $\delta$-SG configuration into a design matrix.

\begin{lem}[SG to design-matrix~\cite{DSW}]
\label{lem:dsw-design-matrix}
Let $v_1,\ldots,v_n\in\R^d$ be a $\delta$-SG configuration, and let $V$ be the $n\times d$ matrix whose $i$th row is $v_i$. Put $K=\lceil\delta(n-1)\rceil$. Then there is an $m\times n$ $(3,3K,6)$-design matrix $A$, every row of which has support of size exactly three, such that $AV=0$.
\end{lem}

We can now prove our first preliminary lemma, which shows that a configuration of points realizing the special lines of a graph $G$ with high minimum degree cannot have almost all of its points on a proper flat. (Notice that this result implies Theorem~\ref{thm:quantitative-sg} by taking $h=0$).
\begin{lem}[Relative Sylvester--Gallai bound]
\label{lem:relative-sg}
Let $P=(p_1,\ldots,p_n)$ be a special-line realization of a graph $G$ with $\delta(G)\geq D\geq 1$, and let $H$ be an affine $h$-flat. Then
\[
        |P \setminus H| \geq
        \frac{D}{12}\bigl(\affdim(P)-h\bigr).
\]
\end{lem}

\begin{proof}
Put $d=\affdim(P)$ and assume w.l.o.g that $h<d$. Replace $H$ by an affine $h$-flat in $\aff(P)$ containing $P\cap H$; this can only decrease $|P\setminus H|$. We may now identify $\aff(P)$ with $\R^d$. After translating by a point of $H$, we may further assume that $H$ is a linear $h$-subspace of $\R^d$.

Put $Q=\{i:p_i\notin H\}$, $k=\lceil D\rceil$, and let $V$ be the matrix whose rows are the vectors $p_i$. Since every $p_i$ has at least $k$ neighbors and every such neighbor lies with $p_i$ on a special line, the points $p_1,\ldots,p_n$ form a $k/(n-1)$-SG configuration. Apply Lemma~\ref{lem:dsw-design-matrix} with $v_i=p_i$, $V$ as above, and $\delta=k/(n-1)$. The lemma gives a $(3,3k,6)$-design matrix $A$ such that $AV=0$.

Let $\pi:\R^d\to\R^d/H$ be the quotient map, and let $X$ be the matrix whose rows are the vectors $\pi(p_i)$, $i\in Q$, in some choice of coordinates on the quotient. Since the vectors $p_i$ span $\R^d$ and the points in $H$ vanish under $\pi$, the rows of $X$ span $\R^d/H$. Hence
\[
\operatorname{rank}(X)=d-h.
\]
In particular, $Q\neq\emptyset$.

Let $A_Q$ be obtained from $A$ by deleting the columns indexed by points in $H$. Applying $\pi$ to $AV=0$ gives $A_QX=0$. Moreover, $A_Q$ is still a $(3,3k,6)$-design matrix: retained column supports and their pairwise intersections are unchanged, while row supports can only decrease. Applying Proposition~\ref{prop:dsw-design-rank} to the matrix $A_Q$, we get
\[
\operatorname{rank}(A_Q)
\geq\frac{|Q|}{1+3\cdot 2\cdot 6/(3k)}
=\frac{|Q|k}{k+12}.
\]
Consequently,
\[
d-h
=\operatorname{rank}(X)
\leq |Q|-\operatorname{rank}(A_Q)
\leq\frac{12|Q|}{k+12}
\leq\frac{12|Q|}{D}.
\]
Since $|Q|=|P\setminus H|$, this completes the proof of Lemma~\ref{lem:relative-sg}.
\end{proof}

\subsection{Avoidance lemmas}

We prove two lemmas on the probability that a random graph avoids the edge set of a given ordinary graph. Here $\widehat G_{n,D}$ is as in Definition~\ref{def:symmetrized-random-out}.

\paragraph{Avoiding a graph in $\R^3$:}

We will need the following theorem of de Zeeuw giving a lower bound on the number of ordinary lines through a configuration in $\R^3$ that does not have too many points on any plane.

\begin{thm}[de Zeeuw {\cite[Theorem~1.3]{deZeeuw}}]
\label{thm:de-zeeuw-ordinary-lines}
For every $0<\alpha<1$, there is $c_\alpha>0$ such that, if a set $P$ of $n$ points in $\R^3$ has at most $\alpha n$ points on any plane, then $P$ spans at least $c_\alpha n^2$ ordinary lines.
\end{thm}

In the following lemma, we quantify the probability of missing all edges in an ordinary graph as a function of the number of points outside a plane.

\begin{lem}[Avoiding an ordinary graph in $\R^3$]
\label{lem:out-graph-avoidance}
Let $n$ be sufficiently large and let $1\leq D\leq n-1$ be an integer. Let $P\subset\mathbb R^3$ be a labeled $n$-point configuration that affinely spans $\mathbb R^3$, and write $O=\Ord(P)$.

\begin{enumerate}
\item If no plane contains more than $2n/3$ points of $P$, then
      \[
          \Pr\bigl(E(\widehat G_{n,D})\cap E(O)=\emptyset\bigr)
          \le \exp(-cDn)
      \]
for an absolute constant $c>0$.

\item Suppose that a plane $H$ contains $n-q>2n/3$ points of $P$. Then
      \[
          \Pr\bigl(E(\widehat G_{n,D})\cap E(O)=\emptyset\bigr)
          \le
          \left(\frac{2q}{n}\right)^{Dq},
      \]
with the left-hand side equal to zero when $D>2q-2$.
\end{enumerate}
\end{lem}

\begin{proof}
\emph{First case.} Apply Theorem~\ref{thm:de-zeeuw-ordinary-lines} with $\alpha=2/3$. It gives
\[
|E(O)|\geq \Omega(n^2).
\]
By independence,
\begin{align*}
\Pr\bigl(E(\widehat G_{n,D})\cap E(O)=\emptyset\bigr)
&=\prod_{v=1}^n\frac{\binom{n-1-d_O(v)}{D}}{\binom{n-1}{D}}\\
&\leq\prod_{v=1}^n\exp\left(-\frac{D\,d_O(v)}{n-1}\right)\\
&=\exp\left(-\frac{2D\,|E(O)|}{n-1}\right)
\leq\exp(-cDn).
\end{align*}

\emph{Second case.} Put $Q=P\setminus H$. Fix $x\in Q$. At most $q-1$ points $y\in P\cap H$ form a special pair with $x$: a third point on $xy$ must lie in $Q\setminus\{x\}$, and each such point determines at most one intersection with $H$.

If $E(\widehat G_{n,D})\cap E(O)=\emptyset$, every member of $N^+(x)$ must therefore lie in a set of size at most
\[
(q-1)+(q-1)=2q-2:
\]
the first term counts $Q\setminus\{x\}$ and the second counts special neighbors in $H$. The probability is zero if $D>2q-2$. Otherwise, again by independence,
\begin{align*}
\Pr\bigl(E(\widehat G_{n,D})\cap E(O)=\emptyset\bigr)
&\leq\prod_{x\in Q}\frac{\binom{2q-2}{D}}{\binom{n-1}{D}}\\
&\leq\left(\frac{2q-2}{n-1}\right)^{Dq}
\leq\left(\frac{2q}{n}\right)^{Dq}.
\end{align*}
\end{proof}

\paragraph{Avoidance in higher dimensions.}
The next lemma combines the relative Sylvester--Gallai bound with the preceding avoidance estimate. It gives a uniform bound for every ordinary graph with a high-dimensional realization.

\begin{lem}[High-dimensional avoidance]
\label{lem:uniform-high-dimensional-avoidance}
There is an absolute constant $b>0$ such that the following holds for all sufficiently large $n$. Let $1\leq D\leq n-1$ and $r\geq 2$ be integers. For every $O\in\Ord(n)_{\geq r+1}$,
\[
\Pr\bigl(E(\widehat G_{n,D})\cap E(O)=\emptyset\bigr)
\leq
\exp\left(-bD^2r\log\left(1+\frac{n}{Dr}\right)\right).
\]
\end{lem}

\begin{proof}
Fix $O\in\Ord(n)_{\geq r+1}$ and choose a realization $P$ of $O$ with affine dimension at least $r+1$. Choose a generic projection $\overline P\subseteq\R^3$ preserving $O$ and affine independence of every affinely independent four-tuple. Since $r\geq 2$, the projected configuration spans $\R^3$. Put
\begin{equation}
\label{eq:uniform-avoidance-T}
T=D^2r\log\left(1+\frac{n}{Dr}\right).
\end{equation}

Suppose first that no plane contains more than $2n/3$ points of $\overline P$ and apply the first case of Lemma~\ref{lem:out-graph-avoidance} to $\overline P$. Since $\log(1+u)\leq u$,~\eqref{eq:uniform-avoidance-T} gives $T\leq Dn$, and the lemma gives the required bound after decreasing $b$.

Suppose instead that a plane $\Pi$ contains $n-q>2n/3$ points of $\overline P$. Apply the second case of Lemma~\ref{lem:out-graph-avoidance} to $\overline P$ to get
\begin{equation}
\label{eq:uniform-avoidance-plane-bound}
\Pr\bigl(E(\widehat G_{n,D})\cap E(O)=\emptyset\bigr)
\leq
\exp\left(-Dq\log\frac{n}{2q}\right),
\end{equation}
and the probability is zero if $D>2q-2$. We may therefore assume that $D\leq 2q-2$ and that there is an outcome for which $E(\widehat G_{n,D})\cap E(O)=\emptyset$. For this outcome, $P$ realizes $\widehat G_{n,D}$.

By the properties of the generic projection, there is an affine flat $H$ of dimension at most two containing exactly $n-q$ points of $P$. Applying Lemma~\ref{lem:relative-sg} and using $\delta(\widehat G_{n,D})\geq D$, gives
\begin{equation}
\label{eq:uniform-avoidance-q-bound}
q\geq\frac D{12}\bigl(\affdim(P)-2\bigr)
\geq\frac{D(r-1)}{12}
\geq\frac{Dr}{24}.
\end{equation}
The function $t\mapsto t\log(1+n/t)$ is increasing on $(0,\infty)$. By~\eqref{eq:uniform-avoidance-q-bound}, and since $q<n/3$ we have
\[
 T = D \cdot Dr\log\left(1+\frac{n}{Dr}\right)
\leq
D \cdot 24q\log\left(1+\frac{n}{24q}\right)
\leq
24Dq\log\frac{n}{2q}.
\]
Combining this with ~\eqref{eq:uniform-avoidance-plane-bound}  proves the required estimate after decreasing $b$.
\end{proof}

\subsection{Proof of Theorem~\ref{thm:main-random-tradeoff}}

Let $B$ be the constant in Lemma~\ref{lem:ordinary-graph-entropy} so that
\[
|\Ord(n)|\leq\exp(Bn\log n),
\]
and let $b>0$ be the constant in Lemma~\ref{lem:uniform-high-dimensional-avoidance}.

\begin{claim}
\label{claim:random-master-bound}
For every integer $r\geq 2$,
\begin{equation}
\label{eq:random-master-bound}
\Pr\bigl(\SGdim(\widehat G_{n,D})>r\bigr)
\leq
\exp\left(Bn\log n-bD^2r\log\left(1+\frac{n}{Dr}\right)\right).
\end{equation}
\end{claim}

\begin{proof}
If $\SGdim(\widehat G_{n,D})>r$, then $\widehat G_{n,D}$ has a special-line realization $P$ with $\affdim(P)\geq r+1$. Put $O=\Ord(P)$. Then $O\in\Ord(n)_{\geq r+1}$ and
\[
E(\widehat G_{n,D})\cap E(O)=\emptyset.
\]
A union bound over $O\in\Ord(n)_{\geq r+1}$, followed by Lemmas~\ref{lem:ordinary-graph-entropy} and~\ref{lem:uniform-high-dimensional-avoidance}, proves the claim.
\end{proof}

We now choose $A\geq 1$ so that
\begin{equation}
\label{eq:random-A-choice}
\frac{bA}{3}\geq B+1,
\end{equation}
and then choose $K>1$ sufficiently large so that
\begin{equation}
\label{eq:random-K-choice}
\frac{t\log t}{2A}\geq\sqrt t
\qquad(t\geq K).
\end{equation}
Put
\begin{equation}
\label{eq:random-x-parameters}
L=\log\left(\frac{D}{\log n}\right),
\qquad
x=\frac{An\log n}{D^2L},
\end{equation}
so that our target dimension bound can be stated as
\begin{equation}
\label{eq:random-r-parameter}
r=\max\{2,\lceil x\rceil\}.
\end{equation}
By~\eqref{eq:random-master-bound} and~\eqref{eq:random-A-choice}, it is enough to prove
\begin{equation}
\label{eq:random-target-exponent}
D^2r\log\left(1+\frac{n}{Dr}\right)
\geq\frac A3n\log n.
\end{equation}

Suppose first that $x\geq 1$. Then
\[
x\leq r\leq 2x.
\]
Moreover,
\[
\frac{n}{Dr}
\geq
\frac{n}{2Dx}
=
\frac{DL}{2A\log n}
\geq
\sqrt{\frac{D}{\log n}},
\]
where the first inequality follows from $r\leq 2x$, the equality follows from~\eqref{eq:random-x-parameters}, and the last inequality follows from~\eqref{eq:random-K-choice} applied with $t=D/\log n$; the hypothesis $D\geq K\log n$ ensures that $t\geq K$. It follows that
\[
\log\left(1+\frac{n}{Dr}\right)
\geq
\log\sqrt{\frac{D}{\log n}}
=
\frac L2.
\]
Since $r\geq x$, the left-hand side of~\eqref{eq:random-target-exponent} is therefore at least
\[
D^2x\frac L2
=
\frac A2n\log n,
\]
which is stronger than~\eqref{eq:random-target-exponent}.

Suppose now that $x<1$. Then $r=2$. Since $L\leq\log n$ for all sufficiently large $n$,~\eqref{eq:random-x-parameters} gives $D^2>An$. The function $t\mapsto t^2\log(1+n/(2t))$ is increasing on $(0,\infty)$, so the left-hand side of~\eqref{eq:random-target-exponent} is at least
\[
2An\log\sqrt{\frac{n}{4A}}
=
An\log\left(\frac{n}{4A}\right)
\geq
\frac A3n\log n
\]
for all sufficiently large $n$. Thus~\eqref{eq:random-target-exponent} holds in both cases. Combining Claim~\ref{claim:random-master-bound} with~\eqref{eq:random-target-exponent} and using the choice of $A$ gives
\[
\Pr\bigl(\SGdim(\widehat G_{n,D})>r\bigr)
\leq\exp(-n\log n).
\]
This completes the proof of Theorem~\ref{thm:main-random-tradeoff}.
\qed

\bibliographystyle{alpha}
\bibliography{GraphSG}

@article{AlonSeymourThomas,
  author  = {Alon, N. and Seymour, P. and Thomas, R.},
  title   = {A separator theorem for nonplanar graphs},
  journal = {J. Amer. Math. Soc.},
  volume  = {3},
  number  = {4},
  pages   = {801--808},
  year    = {1990},
  doi     = {10.1090/S0894-0347-1990-1065053-0}
}

@article{ALWZ,
  author  = {Alweiss, R. and Lovett, S. and Wu, K. and Zhang, J.},
  title   = {Improved bounds for the sunflower lemma},
  journal = {Ann. of Math. (2)},
  volume  = {194},
  number  = {3},
  pages   = {795--815},
  year    = {2021},
  doi     = {10.4007/annals.2021.194.3.5}
}

@inproceedings{BDWY,
  author    = {Barak, B. and Dvir, Z. and Wigderson, A. and Yehudayoff, A.},
  title     = {Rank bounds for design matrices with applications to combinatorial geometry and locally correctable codes},
  booktitle = {Proceedings of the 43rd Annual ACM Symposium on Theory of Computing (STOC 2011)},
  pages     = {519--528},
  year      = {2011},
  doi       = {10.1145/1993636.1993705}
}

@article{BellChueluechaWarnke,
  author  = {Bell, T. and Chueluecha, S. and Warnke, L.},
  title   = {Note on sunflowers},
  journal = {Discrete Math.},
  volume  = {344},
  number  = {7},
  pages   = {Paper No. 112367, 3 pp.},
  year    = {2021},
  doi     = {10.1016/j.disc.2021.112367}
}

@misc{deZeeuw,
  author     = {de Zeeuw, F.},
  shorthand  = {dZ18},
  title      = {Ordinary lines in space},
  year       = {2018},
  eprint     = {1803.09524},
  eprinttype = {arxiv}
}

@article{DSW,
  author  = {Dvir, Z. and Saraf, S. and Wigderson, A.},
  title   = {Improved rank bounds for design matrices and a new proof of {Kelly's} theorem},
  journal = {Forum Math. Sigma},
  volume  = {2},
  pages   = {e4, 24 pp.},
  year    = {2014},
  doi     = {10.1017/fms.2014.2}
}

@article{ErdosRado,
  author  = {Erd{\H{o}}s, P. and Rado, R.},
  title   = {Intersection theorems for systems of sets},
  journal = {J. London Math. Soc.},
  volume  = {35},
  pages   = {85--90},
  year    = {1960},
  doi     = {10.1112/jlms/s1-35.1.85}
}

@article{Frederickson,
  author  = {Frederickson, G. N.},
  title   = {Fast algorithms for shortest paths in planar graphs, with applications},
  journal = {SIAM J. Comput.},
  volume  = {16},
  number  = {6},
  pages   = {1004--1022},
  year    = {1987},
  doi     = {10.1137/0216064}
}

@article{Mader,
  author  = {Mader, W.},
  title   = {Homomorphie\-eigenschaften und mittlere Kanten\-dichte von Gra\-phen},
  journal = {Math. Ann.},
  volume  = {174},
  pages   = {265--268},
  year    = {1967},
  doi     = {10.1007/BF01364272}
}

@article{NesetrilOssonaExpansion,
  author  = {Ne\textcaron{s}et\textcaron{r}il, J. and {Ossona de Mendez}, P.},
  title   = {Grad and classes with bounded expansion {II}. Algorithmic aspects},
  journal = {European J. Combin.},
  volume  = {29},
  number  = {3},
  pages   = {777--791},
  year    = {2008},
  doi     = {10.1016/j.ejc.2006.07.014}
}

@article{ReedWood,
  author  = {Reed, B. A. and Wood, D. R.},
  title   = {A linear-time algorithm to find a separator in a graph excluding a minor},
  journal = {ACM Trans. Algorithms},
  volume  = {5},
  number  = {4},
  pages   = {Article 39, 16 pp.},
  year    = {2009},
  doi     = {10.1145/1597036.1597043}
}

@article{ReidlSanchezStavropoulos,
  author  = {Reidl, F. and {S{\'a}nchez Villaamil}, F. and Stavropoulos, K. S.},
  title   = {Characterising bounded expansion by neighbourhood complexity},
  journal = {European J. Combin.},
  volume  = {75},
  pages   = {152--168},
  year    = {2019},
  doi     = {10.1016/j.ejc.2018.08.001}
}

@article{RonyaiBabaiGanapathy,
  author  = {R{\'o}nyai, L. and Babai, L. and Ganapathy, M. K.},
  title   = {On the number of zero-patterns of a sequence of polynomials},
  journal = {J. Amer. Math. Soc.},
  volume  = {14},
  number  = {3},
  pages   = {717--735},
  year    = {2001},
  doi     = {10.1090/S0894-0347-01-00367-8}
}

@article{StreinuTheran2010,
  author  = {Streinu, I. and Theran, L.},
  title   = {Slider-pinning rigidity: a {Maxwell--Laman}-type theorem},
  journal = {Discrete Comput. Geom.},
  volume  = {44},
  number  = {4},
  pages   = {812--837},
  year    = {2010},
  doi     = {10.1007/s00454-010-9283-y}
}

@article{Whiteley1988,
  author  = {Whiteley, W.},
  title   = {The union of matroids and the rigidity of frameworks},
  journal = {SIAM J. Discrete Math.},
  volume  = {1},
  number  = {2},
  pages   = {237--255},
  year    = {1988},
  doi     = {10.1137/0401025}
}

@article{Whiteley1989,
  author  = {Whiteley, W.},
  title   = {A matroid on hypergraphs, with applications in scene analysis and geometry},
  journal = {Discrete Comput. Geom.},
  volume  = {4},
  number  = {1},
  pages   = {75--95},
  year    = {1989},
  doi     = {10.1007/BF02187716}
}

@incollection{Whiteley1996,
  author    = {Whiteley, W.},
  title     = {Some matroids from discrete applied geometry},
  booktitle = {Matroid Theory (Seattle, WA, 1995)},
  series    = {Contemp. Math.},
  volume    = {197},
  publisher = {Amer. Math. Soc.},
  address   = {Providence, RI},
  pages     = {171--311},
  year      = {1996},
  doi       = {10.1090/conm/197/02540}
}

@article{Yuster,
  author  = {Yuster, R.},
  title   = {Single source shortest paths in {$H$}-minor free graphs},
  journal = {Theoret. Comput. Sci.},
  volume  = {411},
  number  = {34--36},
  pages   = {3042--3047},
  year    = {2010},
  doi     = {10.1016/j.tcs.2010.04.028}
}

\end{document}